\documentclass[12pt,leqno]{article}
\usepackage[dvipsnames]{xcolor}

\usepackage[utf8]{inputenc}
\usepackage{amsmath,amssymb,amsfonts,epsfig,bbm}
\usepackage{amsthm}
\usepackage{stmaryrd,mathabx, mathrsfs}
\usepackage{comment}
\usepackage{color}
\usepackage[T1]{fontenc}

\usepackage{lmodern}
\usepackage{adjustbox}

\usepackage{interval}
\usepackage{enumitem}
\setlist[itemize]{labelindent=*,leftmargin=.7cm}
\setlist[enumerate]{label=(\arabic*),labelindent=*,leftmargin= .7cm}
\usepackage{varwidth}

\usepackage{hyperref}
\hypersetup{
	pdfencoding=auto,
	colorlinks=true,
	citecolor=orange,
	filecolor=black,
	linkcolor=blue,
	urlcolor=blue,
	bookmarksnumbered=true}
\usepackage{csquotes}
\usepackage{mathrsfs}

\usepackage{mathtools}
\DeclarePairedDelimiter\ceil{\lceil}{\rceil}
\DeclarePairedDelimiter\floor{\lfloor}{\rfloor}

\DeclarePairedDelimiter\intbrackets{\llbracket}{\rrbracket}

\newtheorem{defi}{Definition}[section]
\newtheorem{prop}[defi]{Proposition}
\newtheorem{theorem}[defi]{Theorem}

\newtheorem{lemma}[defi]{Lemma}

\newtheorem{remark}[defi]{Remark}

\newcommand{\bdefi}{\begin{defi}}
\newcommand{\edefi}{\end{defi}}
\newcommand{\bprop}{\begin{prop}}
\newcommand{\eprop}{\end{prop}}
\newcommand{\btheo}{\begin{theo}}
\newcommand{\etheo}{\end{theo}}
\newcommand{\btheofr}{\begin{theofr}}
\newcommand{\etheofr}{\end{theofr}}
\newcommand{\blemm}{\begin{lemm}}
\newcommand{\elemm}{\end{lemm}}
\newcommand{\blemmfr}{\begin{lemmfr}}
\newcommand{\elemmfr}{\end{lemmfr}}
\newcommand{\brema}{\begin{rema}}
\newcommand{\erema}{\end{rema}}
\newcommand{\bexer}{\begin{exem}}
\newcommand{\eexer}{\end{exem}}
\newcommand{\bexems}{\begin{exems}}
\newcommand{\eexems}{\end{exems}}
\newcommand{\bconj}{\begin{conj}}
\newcommand{\econj}{\end{conj}}
\newcommand{\bcoro}{\begin{coro}}
\newcommand{\ecoro}{\end{coro}}

\usepackage{mathrsfs}
\renewcommand\mathcal{\mathscr}

\newcommand{\R}{{\cal R}}

\newcommand{\maths}[1]{{\mathbb #1}}  

\newcommand{\CC}{\maths{C}}

\newcommand{\NN}{\maths{N}}

\newcommand{\RR}{\maths{R}}

\newcommand{\ZZ}{\maths{Z}}
\def\11{{\mathbbm 1}}

\newcommand{\weakstar}{\overset{*}\rightharpoonup}

\def\e{\varepsilon}

\newcommand{\bigO}{\operatorname{O}}

\newcommand{\card}{{\operatorname{Card}}}

\newcommand{\Leb}{\operatorname{Leb}}

\newcommand{\supp}{\operatorname{supp}}

\usepackage[
	backend=biber,
	sorting=debug,
	style=numeric-comp
]{biblatex}

\renewbibmacro{in:}{%
	\ifentrytype{article}
	{}
	{\bibstring{in}%
		\printunit{\intitlepunct}}}

\newcommand\numberthis{\addtocounter{equation}{1}\tag{\theequation}}

\usepackage{orcidlink}

\begin{document}
	
\thispagestyle{plain}
\begin{center}
	\Large
	\textbf{Effective pair correlations \\ of fractional powers of integers}

	\large
	\vspace{0.3cm}
	\hspace{0.45cm} Rafael Sayous$^{1,2}$ \orcidlink{0009-0007-6306-8546}
	
	\vspace{0.1cm}
	\today
	
	\vspace{0.3cm}
	{\small $^1$ Laboratoire de Mathématiques d'Orsay, UMR 8628 CNRS,\\
	Universit\'e Paris-Saclay, 91405 ORSAY Cedex, FRANCE.\\
	\vspace{-0.2cm}
	{\it e-mail: rafael.sayous@universite-paris-saclay.fr}}
	
	\vspace{0.2cm}
	{\small $^2$ Department of Mathematics and Statistics, P.O. Box 35,\\
	40014 University of Jyv\"askyl\"a, FINLAND.\\
	\vspace{-0.2cm}
	{\it e-mail: sayousr@jyu.fi}}
\end{center}

\noindent \textbf{Abstract}: We study the statistics of pairs from the sequence $(n^\alpha)_{n\in\NN^*}$, for every parameter $\alpha \in \; ]0,1[\,$. We prove the convergence of the empirical pair correlation measures towards a measure with an explicit density. In particular, when using the scaling factor $N\mapsto N^{1-\alpha}$, we prove that there exists an exotic pair correlation function which exhibits a level repulsion phenomenon. For other scaling factors, we prove that either the pair correlations are Poissonian or there is a total loss of mass. In addition, we give an error term for this convergence.

\vspace{0.2cm}

\vspace{0.2cm}

\medskip
\noindent{\small {\it Acknowledgments:} This research was supported by the French-Finnish CNRS IEA PaCAP. The author would like to thank J.~Parkkonen and F.~Paulin, the supervisors of his ongoing Doctorate, for their support, suggestions and corrections during this research, as well as J.~Marklof, researcher at the University of Bristol, for the discussion and the explanations regarding the unfolding technique.}

\section{Introduction}\label{sec:intro}
In order to understand the distribution of a sequence $(u_n)_{n\in\NN^*}$ in a locally compact metric additive group $G$, an important aspect is the statistics of the spacings between some pairs of points. The approach consisting in taking all pairs of points into account is the study of \emph{pair correlations}, more precisely the asymptotic study of the multisets $F_N=\{ u_n-u_m \}_{1\leq n \neq m \leq N}$ as $N \to \infty$.

These problems were initially developed in physics, especially in quantum chaos, which has lead to a purely mathematical point of view of pair correlations. See for instance \cite{rudnick1998paircorrel_fracpartpoly, aichinger2018quasienergyspectra_pairs, larcher2020pair_maximaladdenergy} for questions directly linked to quantum physics. In various examples for the group $G$, the usual point of comparison for pair correlations is the (almost sure) behaviour of those of a homogeneous Poisson point process of constant intensity on the space $G$. If the pairs from $(u_n)_{n\in\NN^*}$ have the same behaviour, the sequence is said to have Poisson pair correlations. It is of interest on its own to define precisely what this behaviour is and to quantify how "pseudorandom" a deterministic sequence has to be when its pair correlations are Poisson \cite{hinrichs2019poisson_pair_correl_multidim, aistleitner2016pair_correl_equidib, marklof2020pair_correl_equidib_mfd}. Another point of interest is then to find out whether a given sequence has this behavior or not, see the papers \cite{rudnick1998paircorrel_fracpartpoly, boca2005pairs_farey, larcher2018poissonpairs_negativeresults, lutsko2022poissonianslowlygrowingseq, weiss2023explicit}.

For instance, the sequence $(\{n^\alpha\})_{n\in\NN}$, where $\{\cdot\}$ denotes the fractional part function, has Poisson pair correlations if $\alpha$ is small enough, as proven by C.~Lutsko, A.~Sourmelidis and N.~Technau in their paper \cite{lutsko2021pairfracpowers}, and in the special case $\alpha=\frac{1}{2}$, as shown by D.~Elbaz, J.~Marklof and I.~Vinogradov in \cite{elbaz2015two_point_correl_fract_sqrtn_poisson}. As for the pseudorandomness of this sequence, there are two opposite arguments: on one hand, for all $\alpha \in \; ]0,1[\,$, it equidistributes with respect to the Lebesgue measure on $[0,1[$ (see e.g.~\cite[Theo.~2.5]{kuipers1974uniform}), on the other hand, in the case $\alpha=\frac{1}{2}$, it does not behave like a Poisson process at the level of its gaps (i.e.~when we only take into account pairs of points that are consecutive for the order on $[0,1[\, $), as pointed out by N.D.~Elkies and C.T.~McMullen \cite{elkies2004gaps}.

In this paper, the metric group $G$ is $\RR$. Let us give some examples of pair correlations in a noncompact setting. On $G=\RR$, the lengths of closed geodesics in negative curvature have Poisson pair correlation or converge to an exponential probability measure (depending on the scaling factor) \cite{pollicott1998correlpairsclosedgeod, paulinparkko2022b_exponcounting_paircorrel}. On the groups $G=\RR$ then $G=\CC$, the special case where $u_n=\log(n)$ has been shown to exhibit three different behaviours (once again depending on the scaling factor) \cite{paulinparkko2022a_pairs_log, paulinparkko2022c_pairs_complexlog}. Motivated and inspired by these works, we fix $\alpha \in \; ]0,1[\,$ and study the real sequence of general term $u_n=n^\alpha$. Let $\beta \in \; ]0,1[\,$, that we will use as a parameter that determines the scaling. We denote by $\Delta_x$ the unit Dirac mass at $x$. We define the empirical pair correlation measure of $(u_n)_{n\in\NN}$ at order $N$ as
$$\R_N^{\alpha,\beta}=\frac{1}{N^{2-\alpha-\beta}} \sum_{1\leq n \neq m \leq N} \Delta_{N^\beta (n^\alpha-m^\alpha)}.$$
One interesting behaviour in this sum will happen when $n$ and $m$ are close to the upper bound $N$. In that sense, the linear approximation $N^\beta(N^\alpha-(N-1)^\alpha) \sim \alpha N^{\beta-(1-\alpha)}$, as $N\to\infty$, suggests that a fruitful scaling is given by $\beta=1-\alpha$. Such an intuition is confirmed in our main theorem.

In order to state it, we recall that a sequence of positive measures $(\mu_n)_{n\in\NN}$ on $\RR$ is said to \emph{converge vaguely} if there exists a positive measure $\mu$ on $\RR$ such that, for every continuous and compactly supported complex-valued function $f$ defined on $\RR$, we have the convergence $\mu_N(f) \underset{N\to\infty}{\longrightarrow} \mu(f)$, and then we write $\mu_N\underset{N\to\infty}{\weakstar}\mu$. In that context, if $\mu(\RR)<\liminf_{N\to\infty} \mu_N(\RR)$ (resp.~if $\mu(\RR)=0$), we say that the convergence exhibits a \emph{loss of mass} (resp.~\emph{total loss of mass}). If there exists $\e>0$ such that $\mu(\,]-\e,\e[\,)=0$, we say that the measure $\mu$ exhibits a \emph{level repulsion of size $\e$}. Finally, saying that $(n^\alpha)_{n\in\NN^*}$ has \emph{Poisson pair correlations} means that the limit measure $\mu$ has a Radon-Nikodym derivative with respect to the Lebesgue measure which is constant. To illustrate the following theorem, an example of the pair correlation function $\rho_\alpha$ in the case $\beta=1-\alpha$ is shown on Figure \ref{fig:correl_sqrtn_10^6}.

\begin{theorem}\label{th:cv_correlations_simplified}
	We have the following vague convergence of positive measures
	$$ \R_N^{\alpha,\beta} \underset{N \to \infty}{\weakstar} \rho_\alpha \, \Leb_\RR$$
	where $\Leb_\RR$ is the Lebesgue measure on $\RR$ and $\rho_\alpha : \RR \to \RR_+$ is the measurable nonnegative function given by
	$$
	\rho_\alpha : t \mapsto \left\{
	\begin{array}{clc}
		\displaystyle 0 & \mbox{ if } & \beta > 1-\alpha, \vspace{2mm} \\
		
		\displaystyle \frac{1}{\alpha(2-\alpha)} & \mbox{ if } & \beta <1-\alpha, \vspace{2mm} \\
		
		\displaystyle \frac{\alpha^{\frac{1}{1-\alpha}}}{1-\alpha} |t|^{-\frac{2-\alpha}{1-\alpha}} \sum_{p=1}^{\floor*{ \frac{|t|}{\alpha}}} p^{\frac{1}{1-\alpha}} & \mbox{ if } & \beta=1-\alpha,
	\end{array} \right.
	$$
	where $|\cdot|$ denotes the absolute value function on $\RR$, and $\floor*{\cdot}$ is the lower integer part function from $\RR$ to $\ZZ$.
\end{theorem}

We can interpret Theorem \ref{th:cv_correlations} as a result on counting small values in the multisets $F_N=\{n^\alpha - m^\alpha \}_{1\leq n \neq m \leq N}$ as $N \to \infty$. Indeed, the theorem, together with the regularity of the function $\rho_\alpha$, is equivalent to the claim that, for all $a,b \in \RR$ such that $a<b$, we have the convergence
$$\frac{1}{N^{2-\alpha-\beta}} \card \Big( F_N \cap \big]\frac{a}{N^\beta}, \frac{b}{N^\beta}\big[ \Big) \underset{N \to \infty}{\longrightarrow} \int_a^b \rho_\alpha(t) \, dt.$$
Let us comment on the transitional regime $\beta=1-\alpha$. The even function $\rho_\alpha$ is piecewise continuous on $\RR$, with discontinuity at each point in $ \alpha \ZZ-\{0\}$, and bounded: its maximum is reached at the points $\pm \alpha$ and is equal to $\rho_\alpha(\alpha) = \frac{1}{\alpha(1-\alpha)}$. For every $k \in \ZZ- \{0\}$, the function $\rho_\alpha$ is smooth on the open interval $]k \alpha, (k+1) \alpha[$. As $t \to \pm \infty$ in $\RR-\ZZ$, a comparison with an integral shows that $\rho_\alpha\!'(t) \sim \frac{-1}{\alpha (2-\alpha)t}$. Thus the function $\rho_\alpha$ flattens around $\pm \infty$. The same comparison with an integral gives us the convergence $\rho_\alpha \underset{\pm \infty}{\rightarrow} \frac{1}{\alpha(2-\alpha)}$. This limit could be interpreted as a continuity result between the two regimes $\beta=1-\alpha$~and~$\beta<1-\alpha$. Indeed, we have the equality $\supp \R_N^{\alpha,\beta} = N^\beta F_N$, thus the points from the multiset $F_N$ sent to $\pm \infty$ when scaled by a factor $N^\beta$, under the regime $\beta=1-\alpha$, need a smaller scaling factor to be actually observed in the support of a limiting measure: those are points giving rise to the Poisson behaviour of pair correlations of $(n^\alpha)_{n\in\NN^*}$ in the regime $\beta<1-\alpha$. Such a continuity interpretation can also be argued between the cases $\beta=1-\alpha$, for which $\rho_\alpha \Leb_{\RR}$ exhibits a level repulsion of size $\alpha \lambda$, and $\beta>1-\alpha$ where we have a total loss of mass. See Figure \ref{fig:correl_sqrtn_10^6} for an example of both those continuity properties.

\begin{figure}[ht]
	\centering
	\begin{adjustbox}{clip, trim=1.4cm 2.3cm 1.4cm 3.1cm, max width=\textwidth}
		\scalebox{1.}{\includegraphics{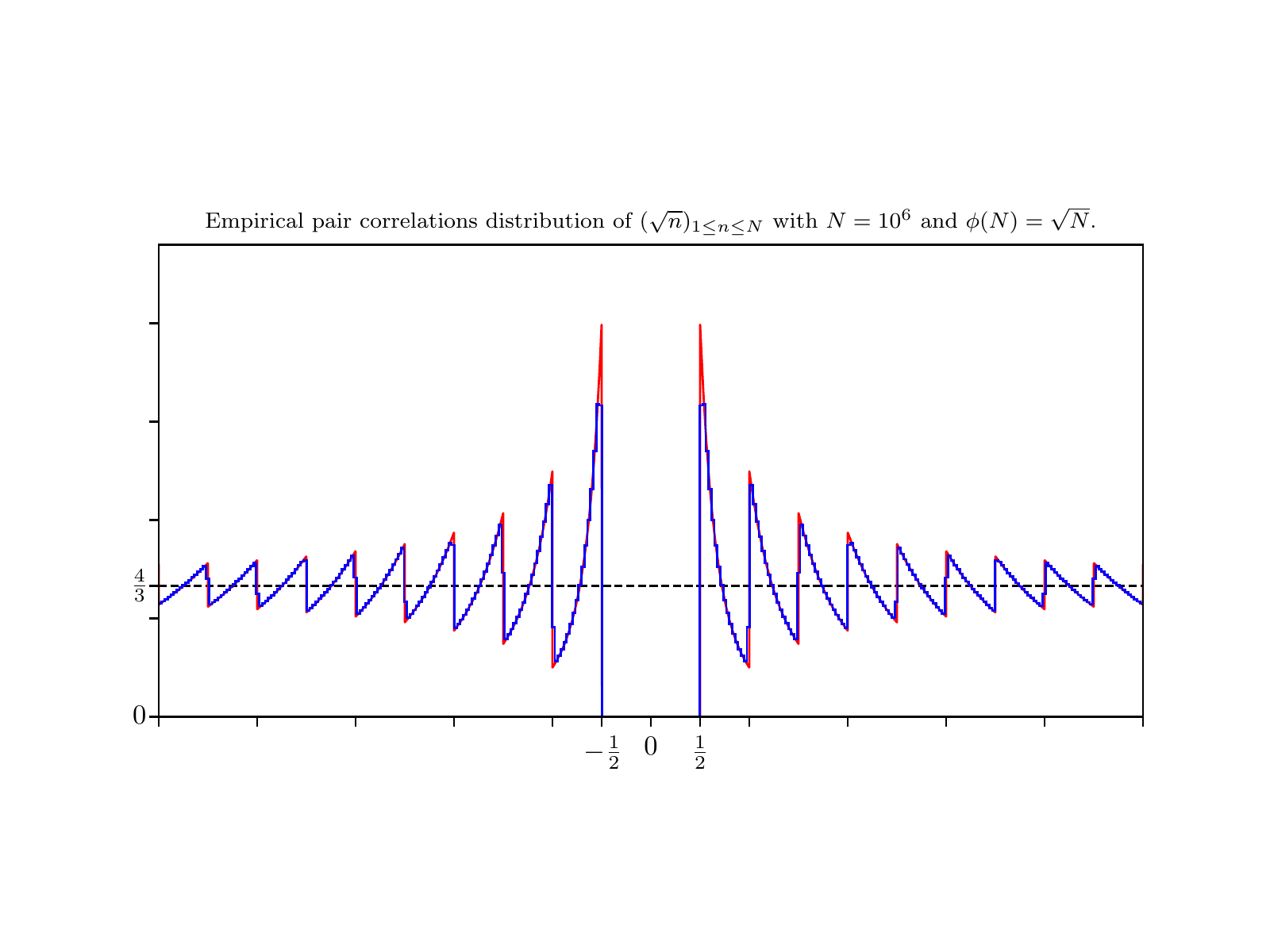}}
	\end{adjustbox}
	\caption{The empirical distribution (in \textcolor{blue}{blue}) of pair correlations for $(\sqrt{n})_{1\leq n\leq N}$ with $N=10^6$ using the scaling factor $N\mapsto \sqrt{N}$ (and renormalization factor $N\mapsto N$), and the limit distribution $\rho_{\frac{1}{2}}$ (in \textcolor{red}{red}).}
	\label{fig:correl_sqrtn_10^6}
\end{figure}

Theorem \ref{th:cv_correlations_simplified} will be stated in more detailed version in Theorem \ref{th:cv_correlations} using a wider range of scaling factors, then in an effective (stronger) version in Theorem \ref{th:effective_cv_correlations}. An interpretation of the pair correlation function $\rho_\alpha$ for $\beta=1-\alpha$ is given in Section \ref{sec:unfolding}.

Our study is much more involved than the work of \cite{paulinparkko2022a_pairs_log} on the pair correlations of $(\log(n))_{n\in\NN^*}$. Here we have different sequences to study in parallel, depending on the parameter $\alpha$. In order to have a precise estimate for the error term, it is important to keep track of its dependence on $\alpha$ in the technical lemmas we use to prove our theorem. The next section is dedicated to that matter.

\section{The main statement and technical lemmas}\label{sec:main_statement_lemmas}

Let $\alpha \in \; ]0,1[$. We will denote the set of nonnegative (resp.~positive) real numbers by $\RR_+$ (resp.~$\RR_+^*$). We are interested in the statistical behaviour of the real sequence $(n^\alpha)_{n\in\NN^*}$. For that purpose, we study its empirical pair correlation measures given by the following general term
$$\R_N^\alpha = \frac{1}{\psi(N)} \sum_{1 \leq n \neq m \leq N} \Delta_{\phi(N) (n^\alpha - m^\alpha)}$$
where for every $x \in \RR$, the notation $\Delta_x$ stands for the Dirac measure at $x$, the function $\phi : \NN \to \RR_+^*$ is called a \emph{scaling factor} and $\psi : \NN \to \RR_+^*$ is called a \emph{renormalization factor}. Both those functions are assumed to be converging to $+\infty$.

\begin{theorem}\label{th:cv_correlations}
	We assume that $\frac{\phi(N)}{N^{1-\alpha}} \underset{N\to\infty}{\longrightarrow} \lambda \in [0,+\infty]$ and for every $N \in \NN$, we set $\psi(N) = \frac{N^{2-\alpha}}{\phi(N)}$. Then, we have the following vague convergence of positive measures
	$$ \R_N^\alpha \underset{N \to \infty}{\weakstar} \rho_\alpha \, \Leb_\RR$$
	where $\Leb_\RR$ is the Lebesgue measure on $\RR$ and $\rho_\alpha : \RR \to \RR_+$ is the measurable nonnegative function given by
	$$
	\rho_\alpha : t \mapsto \left\{
	\begin{array}{clc}
		\displaystyle 0 & \mbox{ if } & \lambda = +\infty, \vspace{2mm} \\
		
		\displaystyle \frac{1}{\alpha(2-\alpha)} & \mbox{ if } & \lambda=0, \vspace{2mm} \\
		
		\displaystyle \frac{\alpha^{\frac{1}{1-\alpha}}}{1-\alpha} \Big(\frac{|t|}{\lambda}\Big)^{-\frac{2-\alpha}{1-\alpha}} \sum_{p=1}^{\floor*{ \frac{|t|}{\alpha \lambda}}} p^{\frac{1}{1-\alpha}} & \mbox{ if } & \lambda \in \RR_+^*.
	\end{array} \right.
	$$
\end{theorem}

We notice that, scaling the pair correlation functions $\rho_\alpha$ in the exotic case $\lambda =1$ for different $\alpha$, we can compare them with each other. Let us define the functions $\widetilde{\rho_\alpha} : t \mapsto \alpha(2-\alpha) \rho_\alpha(\alpha t)$ and see on Figure \ref{fig:correl_sqrtn_diffalpha} how these functions seem to collapse to the null function as $\alpha \to 1$, except at integer points where they explode. This remark can be considered as a continuity observation as $\alpha \to 1$, since a direct computation grants the vague convergence
$$ \frac{1}{N} \sum_{1 \leq n \neq m \leq N} \Delta_{n-m} \underset{N\to\infty}{\weakstar} \sum_{p\in\ZZ^*}\Delta_p.$$

\vspace{-0.55cm}
\begin{figure}[ht]
	\centering
	\begin{adjustbox}{clip, trim=2.5cm 0.2cm 2.2cm 0.8cm, max width=\textwidth}
		\scalebox{1.}{\includegraphics{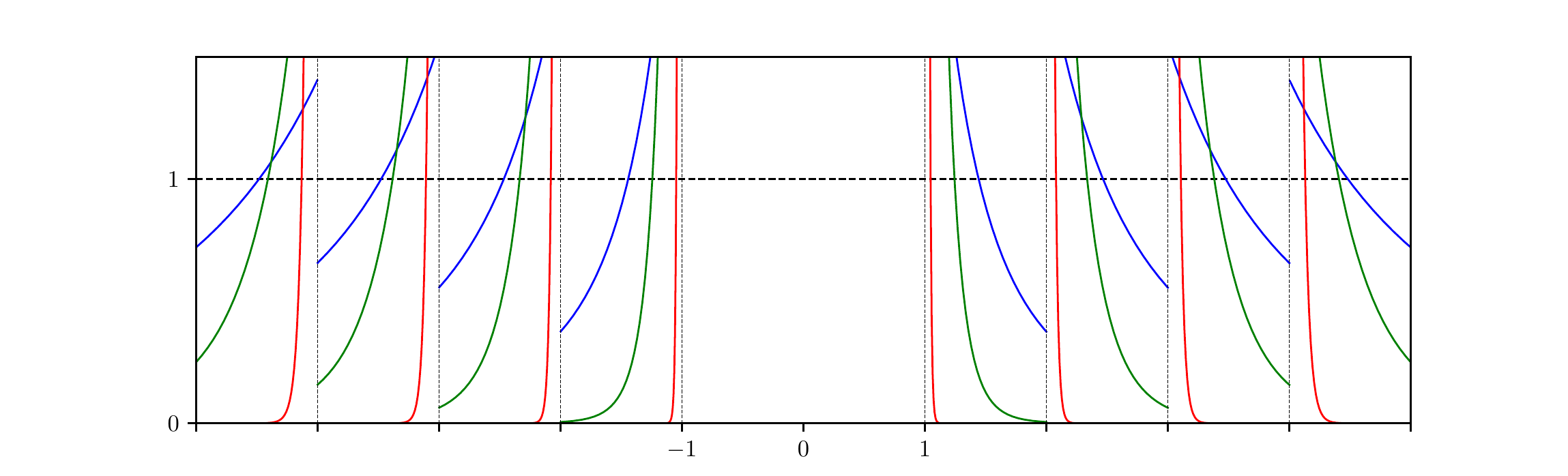}}
	\end{adjustbox}
	\caption{The scaled pair correlation functions $\widetilde{\rho_{\alpha}}$ in the exotic case $\lambda=1$ for different power parameters: $\alpha = \frac{1}{2}$ (in \textcolor{blue}{blue}), $\alpha=\frac{9}{10}$ (in \textcolor{ForestGreen}{green}) and $\alpha=\frac{99}{100}$ (in \textcolor{red}{red}).}
	\label{fig:correl_sqrtn_diffalpha}
\end{figure}

Using $C^1_c$ functions, we also obtain an effective version of Theorem \ref{th:cv_correlations}. To state it, we use Landau's notation. For functions $F,G : \NN \mapsto \CC$ depending on some parameters including $\alpha$, we write $F(N)=\bigO_\alpha(G(N))$ if there exists some constant $c_\alpha >0$, depending only on $\alpha$, and some integer $N_0$, depending on all the parameters, such that, for all $N \geq N_0$, we have the inequality $|F(N)| \leq c_\alpha |G(N)|$. In our study, the rank $N_0$ will depend on the real number $\alpha$, the size $A$ of the support of the test function we evaluate our measures on, and the scaling and renormalization factors.

\begin{theorem}\label{th:effective_cv_correlations}
	We assume that $\frac{\phi(N)}{N^{1-\alpha}} \underset{N\to\infty}{\longrightarrow} \lambda \in [0,+\infty]$ and for every $N \in \NN$, we set $\psi(N) = \frac{N^{2-\alpha}}{\phi(N)}$. Let $f \in C_c^1(\RR)$ and choose $A>1$ such that $\supp f \subset [-A,A]$.
	\begin{itemize}
		\item If $\lambda=+\infty$, then for all $N$ large enough so that $\frac{\alpha \phi(N)}{(2N)^{1-\alpha}} > A$, we have $\R_N^\alpha(f)=0$.
		\item If $\lambda=0$, then there exists $c_\alpha >0$ depending only on $\alpha$ such that, for all $N$ large enough so that $\phi(N) > \frac{A}{2^\alpha-1}$, we have the inequality
		$$
		\big|\R_N^\alpha(f) - \frac{1}{\alpha(2-\alpha)} \Leb_\RR(f)\big| \leq c_\alpha (\|f\|_\infty + \|f'\|_\infty)A^3 \Big(\frac{\phi(N)}{N^{1-\alpha}} +  \frac{1}{N^\alpha \phi(N)} + \frac{1}{N} \Big).
		$$
		\item If $\lambda \in \RR_+^*$, then using the notation $\rho_\alpha$ from Theorem \ref{th:cv_correlations}, we have the estimate
		\begin{align*}
			\R_N^\alpha(f) = & \; \rho_\alpha \Leb_\RR(f) \\
			& + \bigO_\alpha \Big( \frac{A^{\frac{3-2\alpha}{1-\alpha}} (\lambda^2 \|f'\|_\infty+\lambda \|f\|_\infty)}{N} + A^2 \lambda \|f\|_\infty \Big| \Big( \frac{\phi(N)}{\lambda N^{1-\alpha}} \Big)^{\frac{2-\alpha}{1-\alpha}} -1 \Big| \Big).
		\end{align*}
	\end{itemize}
\end{theorem}

\begin{remark}{\rm
An explicit constant $c_\alpha$ will be given at the end of the proof of Theorem \ref{th:effective_cv_correlations} in the case $\lambda=0$. The associated statement gives us a somehow weak control on the error term, as $\frac{\phi(N)}{N^{1-\alpha}}$ can go to zero very slowly. A similar remark applies to the statement regarding the case $\lambda\in\RR_+^*$, since $\frac{\phi(N)}{\lambda N^{1-\alpha}}$ can go to $1$ very slowly.
}
\end{remark}

The fact that Theorem \ref{th:effective_cv_correlations} implies Theorem \ref{th:cv_correlations} comes from the classical argument that one can pass from the convergence of regular measures on $C_c^1$ functions to all $C_c^0$ functions by density for the $\| \cdot \|_\infty$ norm. However, in the space $C_c^0$ we loose any kind of effectiveness as $(\| f'_n \|_\infty)_{n\in\NN}$ can explode along a sequence approximating a continuous function.

\subsection{Symmetry of the empirical pair correlation measures}
For the clarity of the proof, we begin with some practical lemmas. The first one uses the symmetry centered at $0$ of the measures $\R_N^\alpha$. In order to reduce the proof of Theorem \ref{th:effective_cv_correlations} to the asymptotic study of a sequence of measures on $\RR_+$, we define
$$ \R_N^{\alpha,+} = \frac{1}{\psi(N)} \sum_{1 \leq m < n \leq N} \Delta_{\phi(N)(n^\alpha - m^\alpha)} \quad \mbox{and} \quad \R_N^{\alpha,-} = \frac{1}{\psi(N)} \sum_{1 \leq n < m \leq N} \Delta_{\phi(N)(n^\alpha - m^\alpha)} $$
so that we have the decomposition $\R_N^\alpha = \R_N^{\alpha,+} + \R_N^{\alpha,-}$ and the inclusions of their support $\supp(\R_N^{\alpha,+}) \subset \RR_+^*$ and $\supp(\R_N^{\alpha,-}) \subset -\RR_+^*$.

\begin{lemma}\label{lem:correl_symmetry}
	We assume that $\frac{\phi(N)}{N^{1-\alpha}} \underset{N\to\infty}{\longrightarrow} \lambda \in [0,+\infty]$ and for every $N \in \NN$, we set $\psi(N) = \frac{N^{2-\alpha}}{\phi(N)}$. Let $f \in C_c^1(\RR)$ and $A>1$ such that $\supp f \subset [-A,A]$. Set $\check{f}:t\mapsto f(-t)$. Let $F:\NN \to \RR_+$ be a function possibly depending on the parameters $\alpha$, $\phi$, $\|f\|_\infty$, $\|f'\|_\infty$, $A$ and $\lambda$. We assume there exists a real number $c_\alpha>0$, only depending on $\alpha$, and a integer $N_0\in\NN$ such that, for all $N \geq N_0$,
	$$ \big|\R_N^{\alpha,+}(f) - \11_{\RR_+} \rho_\alpha \Leb_\RR(f)\big| \leq \frac{c_\alpha}{2} F(N) \,\mbox{ and }\, \big|\R_N^{\alpha,+}(\check{f}) - \11_{\RR_+} \rho_\alpha \Leb_\RR(\check{f})\big| \leq \frac{c_\alpha}{2} F(N).$$
	Then, for all $N \geq N_0$, we have the inequality
	$$ \big|\R_N^\alpha(f) - \rho_\alpha \Leb_\RR(f)\big| \leq c_\alpha F(N).$$
\end{lemma}

\begin{proof}
	Using the symmetry $\R_N^{\alpha,-} = (t \mapsto -t)_* \R_N^{\alpha,+}$, the invariance of the parameters of $F$ under this change of variable, and the fact that $\rho_\alpha$ is even, we have the inequality, for all $N\geq N_0$,
	 \begin{align*}
	 	\big|\R_N^{\alpha,-}(f) - \11_{\RR_-} \rho_\alpha \Leb_\RR(f)\big| & = \big|(t \mapsto -t)_*(\R_N^{\alpha,+} - \11_{\RR_+} \rho_\alpha \Leb_\RR )(f)\big| \\
	 	& = \big|\R_N^{\alpha,+}(\check{f}) - \11_{\RR_+} \rho_\alpha \Leb_\RR(\check{f})\big| \leq c_\alpha F(N).
	 \end{align*}
	The result then follows from the triangle inequality.
\end{proof}

\subsection{Linear approximation}
The second lemma is a linear approximation process. Indeed, we will be able to approximate the re-written expression $\R_N^{\alpha,+} = \frac{1}{\psi(N)} \sum_{m=1}^{N-1} \sum_{p=1}^{N-m} \Delta_{\phi(N)\left((m+p)^\alpha - m^\alpha\right)}$ by the positive measure on $\RR_+$ defined by
$$ \mu_N^+ = \frac{1}{\psi(N)} \sum_{m=1}^{N-1} \sum_{p=1}^{N-m} \Delta_{\phi(N)\frac{\alpha p }{m^{1-\alpha}}}.$$

\begin{lemma}\label{lem:linear_approx}
	Let $f\in C_c^1(\RR)$ and choose $A>1$ such that $\supp{f} \subset [-A,A]$. Let $N \in \NN^*$. We assume that $N$ is large enough so that $\phi(N) > \frac{A}{2^\alpha-1}$. Then there exists a positive constant $c_\alpha'>0$ depending only on $\alpha$, such that
	$$|\R_N^{\alpha,+}(f) - \mu_N^+(f)| \leq c_\alpha' A^3 \|f'\|_\infty \frac{N^{2(1-\alpha)}}{\psi(N) \phi(N)^2}.$$
\end{lemma}

\begin{remark}\label{rmk:linear_approx}{\rm
	In the case at hand, we will assume that the renormalization factor is linked to the scaling factor by the formula $\psi(N)=\frac{N^{2-\alpha}}{\phi(N)}$. The inequality in Lemma \ref{lem:linear_approx} thus becomes
	$$|\R_N^{\alpha,+}(f) - \mu_N^+(f)| \leq c_\alpha' A^3 \|f'\|_\infty \frac{1}{N^\alpha \phi(N)}.$$}
\end{remark}

\begin{proof}
	For all $a,b \in \ZZ$, we use the double brackets notation $\intbrackets*{a,b}=\{a,a+1,\ldots,b\}$ for the interval of integers between $a$ and $b$. Let $N \geq 2$ (for $N=1$, we have $\R_1^{\alpha,+}=\mu_1^+=0$). For $m \in \llbracket 1, N-1 \rrbracket$ and $p\in \llbracket 1, N-m \rrbracket$, we want to bound from above the quantity $$\big|f\big(\phi(N)((m+p)^\alpha-m^\alpha)\big)-f\big(\phi(N)\frac{\alpha p}{m^{1-\alpha}}\big)\big|.$$
	To do so, we first observe that a bound on the contributing $p$ arises from the fact that the function $f$ is compactly supported. Indeed, for all $m \in \llbracket 1, N-1 \rrbracket$ and $p\in \llbracket 1, N-m \rrbracket$, we have the equivalences
	$$ \begin{array}{rcl}
		\phi(N) \frac{\alpha p}{m^{1-\alpha}} \leq A & \iff & p \leq \frac{A m^{1-\alpha}}{\alpha \phi(N)} \\
		\mbox{ and  } \phi(N)((m+p)^\alpha - m^\alpha) \leq A & \iff & p \leq (\frac{A}{\phi(N)}+m^\alpha)^{\frac{1}{\alpha}}-m.
	\end{array}$$
	Thus we set some bound for $p$ in this proof by defining the function
	\begin{equation}\label{eq:def_pmax} p_{\max} : (N,m) \mapsto \Big\lfloor \max\Big\{ \frac{A m^{1-\alpha}}{\alpha\phi(N)}, \, \big( \frac{A}{\phi(N)}+m^\alpha \big)^{\frac{1}{\alpha}} - m \Big\} \Big\rfloor
	\end{equation}
	and we denote by $I_N$ the set of indices $(m,p)$ respecting that bound on $p$, that is to say $I_N = \{ (m,p) \in \NN^2 \, : \, 1 \leq m \leq N-1, \, 1 \leq p \leq p_{\max}(N,m) \}.$
	
	Applying the mean value inequality to $f$ and the Taylor-Lagrange inequality to the function $x \mapsto (1+x)^\alpha$, we obtain the following inequality, for all $(m,p) \in I_N$,
	\begin{align*}
	\big|f(\phi(N)\left((m+p)^\alpha-m^\alpha)\right)-f\big(\phi(N)\frac{\alpha p}{m^{1-\alpha}}\big)\big| & \leq \|f'\|_\infty \phi(N) m^\alpha \big| \big(1+\frac{p}{m}\big)^\alpha - 1 - \frac{\alpha p}{m} \big| \\
	& \leq \|f'\|_\infty \phi(N) \frac{\alpha(1-\alpha)}{2} \frac{p^2}{m^{2-\alpha}}. \numberthis\label{eq:taylor_approx_m,p}
	\end{align*} 
	
	Our goal is then to bound from above the sum $\sum_{(m,p)\in I_N} \frac{p^2}{m^{2-\alpha}}$. For that purpose, we use some integral comparison. We extend $p_{\max}$ to $\NN^* \times \RR_+$ still using the expression \eqref{eq:def_pmax}. We will compare the above sum to the integral defined by
	$$ J_N = \int_{x=1}^{N-1} \int_{y=1}^{p_{\max}(N,x+1)+1} \frac{y^2}{x^{2-\alpha}} \, dy dx.$$
	
	To justify the comparison, we begin with a unit square. The variations in each variable of the integrand function provide the inequality, for every $m,p \in \NN^*$,
	$$
	\int_{[m,m+1] \times [p,p+1]} \frac{y^2}{x^{2-\alpha}} \, dy dx \geq \frac{p^2}{(m+1)^{2-\alpha}}.
	$$
	We can bound from below the integral $J_N$ by the sum of integrals on the unit squares under the graph of the nondecreasing function $p_{\max}(N,\cdot)+1$. We thus obtain
	\begin{align*}
	\int_{x=1}^{N-1} \int_{y=1}^{p_{\max}(N,x+1)+1} \frac{y^2}{x^{2-\alpha}} \, dy dx & \geq \sum_{m=1}^{N-2} \; \sum_{p=1}^{p_{\max}(N,m+1)} \int_{[m,m+1] \times [p,p+1]} \frac{y^2}{x^{2-\alpha}} \, dy dx \\
	& \geq \sum_{m=1}^{N-2} \; \sum_{p=1}^{p_{\max}(N,m+1)} \frac{p^2}{(m+1)^{2-\alpha}} \\
	& = \sum_{(m,p) \in I_N} \frac{p^2}{m^{2-\alpha}} - \sum_{p=1}^{p_{\max}(N,1)} p^2.
	\end{align*}
	As $\phi \underset{\infty}{\longrightarrow}+\infty$, the definition of $p_{\max}$ indicates that $p_{\max}(N,1)=0$ for $N$ large enough. More precisely, it is the case if we have both inequalities $\frac{A}{\alpha \phi(N)}<1$ and $( \frac{A}{\phi(N)}-1)^{\frac{1}{\alpha}}-1<1$, or equivalently if we have $\phi(N) > A \max\big\{\frac{1}{\alpha}, \frac{1}{2^\alpha -1}\big\}= \frac{A}{2^\alpha-1}$ which is one of our assumptions. Hence we have the inequality $\sum_{(m,p)\in I_N} \frac{p^2}{m^{2-\alpha}} \leq J_N$.
	
	It remains to evaluate the integral $J_N$. Using the facts that for all $x \geq 1$ (or $x=0$), we have the inequality $(x+1)^3-1 \leq 2^3 x^3$ and by observing that $p_{\max}(N,\cdot)$ is integer-valued, we obtain the following sequence of inequalities
	\begin{align*}
			J_N & =  \int_1^{N-1} \int_1^{p_{\max}(N,x+1)+1} \frac{y^2}{x^{2-\alpha}} \, dy dx = \frac{1}{3} \int_1^{N-1} \frac{(p_{\max}(N,x+1)+1)^3-1 }{x^{2-\alpha}} \, dx  \\[1.5mm]
			& \leq \frac{2^3}{3} \int_{1}^{N-1} \frac{p_{\max}(N,x+1)^3}{x^{2-\alpha}} \, dx \\[1.5mm]
			& \leq \frac{8}{3} \int_{1}^{N-1} \frac{\max\big\{ \frac{A(x+1)^{1-\alpha}}{\alpha \phi(N)}, (\frac{A}{\phi(N)}+(x+1)^\alpha)^{\frac{1}{\alpha}}-(x+1) \big\}^3}{x^{2-\alpha}} \, dx \\[1.5mm]
			& \leq \frac{8}{3} \frac{2^{3(1-\alpha)}A^3}{\alpha^3 \phi(N)^3} \int_1^{N-1} x^{1-2\alpha} \, dx + \frac{8}{3} \int_1^{N-1} \frac{(x+1)^3 \big((\frac{A}{\phi(N)(x+1)^\alpha}+1)^{\frac{1}{\alpha}}-1 \big)^3}{x^{2-\alpha}} \, dx \\[1.5mm]
			& \leq \frac{8.2^{3(1-\alpha)}A^3}{ 6(1-\alpha)\alpha^3} \frac{N^{2(1-\alpha)}}{\phi(N)^3} + \frac{8}{3} J_N',
	\end{align*}
	where $J_N'$ is the last integral on the previous line. Using the mean value inequality for the map $x \mapsto (1+x)^\frac{1}{\alpha}$ (whose derivative is increasing) between $0$ and $\frac{A}{\phi(N)(x+1)^\alpha}$, and the inequality $\frac{A}{\phi(N)} \leq 1$ (coming from our assumption $\phi(N) > \frac{A}{2^\alpha-1}$), the remaining integral $J_N'$ can be bounded as follows
	\begin{align*}
		J_N' & \leq \int_1^{N-1} \frac{(x+1)^3 \big(	\frac{1}{\alpha} (1+\frac{A}{\phi(N)(x+1)^\alpha})^{\frac{1}{\alpha}-1}\frac{A}{\phi(N)(x+1)^\alpha} \big)^3 }{x^{2-\alpha}} \, dx \\[1.5mm]
		& \leq 	\frac{2^{\frac{1}{\alpha}-1}A^3}{\alpha^3} \frac{1}{\phi(N)^3} \int_1^{N-1} \frac{(x+1)^{3(1-\alpha)}}{x^{2-\alpha}} \, dx \\[1.5mm]
		& \leq \frac{2^{\frac{1}{\alpha}+2-3\alpha}A^3}{\alpha^3} \frac{1}{\phi(N)^3} \int_1^{N-1} x^{1-2\alpha} \, dx \leq 	\frac{2^{\frac{1}{\alpha}+2-3\alpha}A^3}{\alpha^3} \frac{N^{2(1-\alpha)}}{\phi(N)^3}.
	\end{align*}
	Combining this integral approximation with our inequality \eqref{eq:taylor_approx_m,p}, we finally obtain the inequality
	\begin{align*}
		& |\R_N^{\alpha,+}(f) - \mu_N^+(f)| \\[1.5mm]
		\leq & \frac{\alpha(1-\alpha)}{2} \|f'\|_\infty \frac{\phi(N)}{\psi(N)} \sum_{(m,p) \in I_N} \frac{p^2}{m^{2-\alpha}}
		\leq \frac{\alpha(1-\alpha)}{2} \|f'\|_\infty 	\frac{\phi(N)}{\psi(N)} J_N \\[1.5mm]
		\leq & \frac{\alpha(1-\alpha)}{2} \|f'\|_\infty 	\frac{\phi(N)}{\psi(N)} \Big( \frac{8\cdot2^{3(1-\alpha)}A^3}{6(1-\alpha)\alpha^3} \frac{N^{2(1-\alpha)}}{\phi(N)^3} + \frac{8\cdot 2^{\frac{1}{\alpha}+2-3\alpha}A^3}{3\alpha^3} \frac{N^{2(1-\alpha)}}{\phi(N)^3} \Big) \\[1.5mm]
		\leq & c_\alpha' A^3 \|f'\|_\infty 	\frac{N^{2(1-\alpha)}}{\psi(N) \phi(N)^2}
	\end{align*}
	where $c_\alpha'=\frac{2^{1+3(1-\alpha)}+2^{\frac{1}{\alpha}+4-3\alpha}(1-\alpha)}{3\alpha^2} \leq \frac{32}{3}\frac{2^\frac{1}{\alpha}}{\alpha^2}$.
\end{proof}

\subsection{Riemann sum approximation for compactly supported functions}

Finally, the third lemma is a practical quite standard version of the Riemann sum approximation with estimate of the error term which is suitable for compactly supported $C^1$ functions.

\begin{lemma}\label{lem:riemann_approx}
	Let $f \in C_c^1(\RR)$ and choose $B \geq 0$ such that $\supp{f} \subset [-B,B]$. Let $\delta >0$ and $M \in \NN^*$. Then
	$$\Big| \int_0^{M\delta} f(t)dt - \delta \sum_{p=1}^M f(p\delta) \Big| \leq \frac{\|f'\|_\infty}{2} \delta \min\{B,M\delta\}.$$
\end{lemma}

\begin{proof}
	Assume that $M\delta \leq B$. By the triangle and mean value inequalities, we thus have, for all $p \in \llbracket 1,M \rrbracket$, 
	$$
	\Big| \int_{(p-1)\delta}^{p\delta} f(t)dt - \delta f(p\delta) \Big| \leq \int_{(p-1)\delta}^{p\delta} |f(t) - f(p\delta)| dt \leq \int_{(p-1)\delta}^{p\delta} \|f'\|_\infty (p\delta-t)dt = \|f'\|_\infty \frac{\delta^2}{2}.
	$$
	By summing over $p \in \llbracket 1,M \rrbracket$ and using the triangle inequality, the lemma is proved in the case $M\delta \leq B$.
	
	Now let us assume that $M\delta > B$. The quantity we want to evaluate can be written
	$$\Big| \int_0^{B} f(t)dt - \delta \sum_{p=1}^{\floor*{\frac{B}{\delta}}} f(p\delta) \Big|.$$
	The case we first proved thus yields the inequality
	\begin{equation}\label{eq:riemann_approx_1}
	\Big| \int_0^{\delta \floor*{\frac{B}{\delta}}} f(t)dt - \delta \sum_{p=1}^{\floor*{\frac{B}{\delta}}} f(p\delta) \Big| \leq \frac{\|f'\|_\infty}{2}\delta^2 \Big\lfloor \frac{B}{\delta} \Big\rfloor.
	\end{equation}
	For the remaining part of the integral, we use once again the triangle and mean value inequalities and obtain
	\begin{equation}\label{eq:riemann_approx_2}
		 \Big| \int_{\delta \floor*{\frac{B}{\delta}}}^B f(t)dt \Big| \leq \int_{\delta \floor*{\frac{B}{\delta}}}^B |f(t)-f(B)| dt
		 \leq \frac{\|f'\|_\infty}{2}\Big(B-\delta \Big\lfloor \frac{B}{\delta} \Big\rfloor \Big)^2
	\end{equation}
	Summing both inequalities \eqref{eq:riemann_approx_1} and \eqref{eq:riemann_approx_2}, we get
	\begin{align*}
		\Big| \int_0^{B} f(t)dt - \delta \sum_{p=1}^{\floor*{\frac{B}{\delta}}} f(p\delta) \Big| & \leq \frac{\|f'\|_\infty}{2} \delta^2 \Big( \Big\lfloor \frac{B}{\delta} \Big\rfloor + \Big(\frac{B}{\delta} -  \Big\lfloor \frac{B}{\delta} \Big\rfloor \Big)^2 \Big) \\
		& \leq \frac{\|f'\|_\infty}{2} \delta^2 \Big( \Big\lfloor \frac{B}{\delta} \Big\rfloor + \Big(\frac{B}{\delta} -  \Big\lfloor \frac{B}{\delta} \Big\rfloor\Big) \Big) = \frac{\|f'\|_\infty}{2} \delta B.
	\end{align*}
This concludes the proof of the Lemma \ref{lem:riemann_approx}.
\end{proof}

\section{Proof of Theorem \ref{th:effective_cv_correlations}}\label{sec:proof_thm}
We now have the tools to prove our theorem.
As we are studying three regimes for the scaling factor $\phi$ that are completely different in terms of behavior of the sequence $(\R_N^\alpha)_{N\in\NN}$, the proof will be divided accordingly. Recall that we impose, for all $N \in \NN$, that the renormalization factor is $\psi(N)=\frac{N^{2-\alpha}}{\phi(N)}$, even though it has no importance in the first regime. By Lemma \ref{lem:correl_symmetry}, we only need to study the effective behavior of the positive part of our pair correlation measures, which is defined by
$$\R_N^{\alpha,+}=\frac{1}{\psi(N)} \sum_{m=1}^{N-1} \sum_{p=1}^{N-m} \Delta_{\phi(N)((m+p)^\alpha-m^\alpha)}.$$
Let $f \in C_c^1(\RR)$ and choose $A > 1$ such that $\supp{f} \subset [-A, A]$.
	\medskip

\subsection[regime +infini]{Regime $\frac{\phi(N)}{N^{1-\alpha}} \underset{N \to \infty}{\longrightarrow}+\infty$}
In this first case, we want to show the vague convergence towards $0$. For all $x \geq 0$, we have the inequality
$$(1+x)^\alpha - 1 = \int_0^x \alpha (1+t)^{\alpha-1}dt \geq \frac{\alpha x}{(1+x)^{1-\alpha}}.$$
Consequently, for all $N\in\NN^*$, all $m\in \llbracket 1,N-1 \rrbracket$ and all $p\in\llbracket 1, N-m \rrbracket$, we obtain
\begin{align*}
	\phi(N)((m+p)^\alpha-m^\alpha) & = \phi(N) m^\alpha \big(\big(1+\frac{p}{m}\big)^\alpha-1\big) \\
	& \geq \phi(N) m^\alpha \frac{\alpha \frac{p}{m}}{(1+\frac{p}{m})^{1-\alpha}} = \phi(N) \frac{\alpha p}{(m+p)^{1-\alpha}} \geq \frac{\alpha \phi(N)}{(2N)^{1-\alpha}}. \numberthis\label{eq:empirical_level_repulsion}
\end{align*}
	One can notice that we have not yet used any assumption on $\phi$ (other than its positivity). If $N$ is large enough so that $\frac{\alpha \phi(N)}{(2N)^{1-\alpha}} > A$, Equation \eqref{eq:empirical_level_repulsion} yields the equality $\R_N^{\alpha,+}(f)=0$ (in fact, independently on the choice of the renormalization factor $\psi$). That concludes the proof in the first case.
	\medskip
	
\subsection[regime 0]{Regime $\frac{\phi(N)}{N^{1-\alpha}} \underset{N\to\infty}{\longrightarrow}0$}
	
Our goal is to show the asymptotic Poissonian behaviour of $(\R_N^{\alpha,+})_{N\in\NN}$, with the speed of convergence described in Theorem \ref{th:effective_cv_correlations}. By Lemma \ref{lem:linear_approx} (more precisely, by the Remark \ref{rmk:linear_approx} following it), it suffices to prove the same result for $(\mu_N^+)_{N\in\NN}$ instead. As we want to show some convergence towards a measure absolutely continuous with respect to the Lebesgue measure $\Leb_{\RR_+}$ on $\RR_+$, we will use a Riemann sum approximation of the sums defining the measures $\mu_N^+$ and thus compare them to integrals. For that matter, for all $N\in\NN-\{0\}$ and $m\in\llbracket 1,N-1 \rrbracket$, we set $$\delta_{N,m}=\frac{\alpha\phi(N)}{m^{1-\alpha}},$$ corresponding to the step appearing in the second sum defining $\mu_N^+$. Let $N\in\NN-\{0\}$. We define the positive measure $\nu_N^+$ on $\RR_+$ by
$$\nu_N^+ = \frac{1}{\psi(N)} \sum_{m=1}^{N-1} \frac{1}{\delta_{N,m}} \11_{[0,(N-m)\delta_{N,m}]} \Leb_{\RR_+}.$$
Lemma \ref{lem:riemann_approx} with $B=A$, $M=N-m$ and $\delta=\delta_{N,m}$ grants us the inequality
\begin{align*}
	|\mu_N^+(f)-\nu_N^+(f)| & = \frac{1}{\psi(N)} \Big| \sum_{m=1}^{N-1} \frac{1}{\delta_{N,m}} \Big( \delta_{N,m} \sum_{p=1}^{(N-m)} f(p\delta_{N,m}) - \int_0^{(N-m)\delta_{N,m}} f(t)dt \Big) \Big| \\
	& \leq \frac{1}{\psi(N)} \sum_{m=1}^{N-1} \frac{1}{\delta_{N,m}} \frac{\|f'\|_\infty}{2} \delta_{N,m} \min\{A,(N-m)\delta_{N,m}\}. \numberthis\label{eq:error_mu_N_nu_N_intermediate}
\end{align*}
In order to evaluate the above sum of such minima, we use the following equivalence, for all $m\in\llbracket 1,N-1 \rrbracket$,
$$ (N-m)\delta_{N,m} \leq A \iff g_N(m) \leq 0 \mbox{ where } g_N : x \mapsto (N-x) - \frac{A}{\alpha\phi(N)}x^{1-\alpha}.$$
A straightforward study of the functions $g_N$ shows that they each admit only one zero $x_N$ in $]0,N[$, which has the asymptotic behavior $x_N \sim N$. More precisely, we have $N(1-\frac{A}{\alpha N^\alpha \phi(N)})= N-\frac{AN^{1-\alpha}}{\alpha \phi(N)} \leq x_N \leq N.$
Using the inequality \eqref{eq:error_mu_N_nu_N_intermediate}, we obtain
\begin{align*}
	|\mu_N^+(f)-\nu_N^+(f)| & \leq \frac{\|f'\|_\infty}{2\psi(N)} \Big( \sum_{m=1}^{\floor*{x_N}} A
	+ \alpha\phi(N) \sum_{m=\floor*{x_N}+1}^{N-1} \frac{N-m}{m^{1-\alpha}} \Big) \\
	& \leq \frac{\|f'\|_\infty}{2\psi(N)} \Big(NA + \alpha \phi(N) \int_{x_N}^N \frac{N-x}{x^{1-\alpha}}\,dx \Big)
\end{align*}
since $x \mapsto \frac{N-x}{x^{1-\alpha}}$ is nonincreasing on $]0,N]$. Using the above approximation of $x_N$, we get, for all $x \in [x_N,N]$, the inequality $N-x \leq \frac{AN^{1-\alpha}}{\alpha\phi(N)}$. The integral $\int_{x_N}^N \frac{N-x}{x^{1-\alpha}}\,dx$ is then bounded from above by $\frac{AN}{\alpha^2\phi(N)}$. Since $\psi(N)=\frac{N^{2-\alpha}}{\phi(N)}$, it yields
\begin{equation}\label{eq:error_mu_N_nu_N}
	|\mu_N^+(f)-\nu_N^+(f)| \leq \frac{\|f'\|_\infty}{2\psi(N)} \Big( NA + \alpha \phi(N) \frac{AN}{\alpha^2\phi(N)} \Big) \leq \frac{(1+\alpha) A\|f'\|_\infty}{2\alpha} \frac{\phi(N)}{N^{1-\alpha}}.
\end{equation}
We remark that this error term goes to $0$ only in the case at hand: we won't be able to use the same measures $\mu_N^+$ and $\nu_N^+$ for the last case $\frac{\phi(N)}{N^{1-\alpha}} \underset{N\to\infty}{\longrightarrow}\lambda \in \RR_+^*$. Now that we are assured that the measure $\nu_N^+$ is a good approximation of $\mu_N^+$, we can move forward and study the convergence of $(\nu_N^+)_{N\in\NN}$. As those measures have a density, that we denote by $\theta_N$, with respect to the Lebesgue measure, we study their pointwise convergence. Let $t\in\RR_+$. We have
$$\theta_N(t) = \frac{1}{\psi(N)} \sum_{m=1}^{N-1} \frac{1}{\delta_{N,m}}\11_{[0,(N-m)\delta_{N,m}]}(t).$$
To see its behavior as $N\to\infty$, we use a $t$-depending version of the function $g_N$ : for all $m\in\intbrackets*{1,N-1}$, we have
$$ (N-m)\delta_{N,m} \leq t \iff g_{N,t}(m)\leq 0 \mbox{ where } g_{N,t}:x\mapsto (N-x) - \frac{t}{\alpha\phi(N)}x^{1-\alpha}.$$
Once again, each $g_{N,t}$ only has one zero in $\RR_+$ that we denote by $x_{N,t}$, and we still have the approximation $N(1-\frac{t}{\alpha N^\alpha\phi(N)})=N-\frac{tN^{1-\alpha}}{\alpha\phi(N)} \leq x_{N,t} \leq N$. Since $\psi(N)=\frac{N^{2-\alpha}}{\phi(N)}$, we can rewrite $\theta_N(t)$ as follows:
\begin{equation}\label{eq:theta_N(t)_sum}
	\theta_N(t) = \frac{1}{\psi(N)} \sum_{m=1}^{\floor*{x_{N,t}}}\frac{1}{\delta_{N,m}} = \frac{1}{\alpha\phi(N)\psi(N)} \sum_{m=1}^{\floor*{x_{N,t}}} m^{1-\alpha}=\frac{1}{\alpha N^{2-\alpha}} \sum_{m=1}^{\floor*{x_{N,t}}} m^{1-\alpha}.
\end{equation}
The last sum is comparable to an integral. More precisely, we have the approximation
\begin{align*}
	& \int_0^{\floor{x_{N,t}}}x^{1-\alpha}dx \leq \sum_{m=1}^{\floor*{x_{N,t}}} m^{1-\alpha} \leq \int_0^{\floor{x_{N,t}}}x^{1-\alpha}dx + \floor{x_{N,t}}^{1-\alpha} \\
	\mbox{i.e. } & \frac{1}{2-\alpha}\floor*{x_{N,t}}^{2-\alpha} \leq \sum_{m=1}^{\floor*{x_{N,t}}} m^{1-\alpha} \leq \frac{1}{2-\alpha}\floor*{x_{N,t}}^{2-\alpha} + \floor{x_{N,t}}^{1-\alpha}.
\end{align*}
Combining this integral comparison with the expression \eqref{eq:theta_N(t)_sum} and using the asymptotic behavior $x_{N,t} \sim N$ as $t \in \RR_+$ is fixed, we get the pointwise convergence
\begin{equation}\label{eq:theta_N(t)_limit}
	\theta_N(t) \underset{N\to\infty}{\longrightarrow} \frac{1}{\alpha(2-\alpha)}=\rho_\alpha(t).
\end{equation}
We could conclude the proof of Theorem \ref{th:cv_correlations} in the case at hand, that is under the regime $\frac{\phi(N)}{N^{1-\alpha}}\underset{N\to\infty}{\longrightarrow}0$, by use of the dominated convergence theorem. However, for the effective version we present in Theorem \ref{th:effective_cv_correlations}, we need more precision. First, we have the inequality
$$|\nu_N^+(f) - \rho_\alpha \Leb_{\RR_+}(f)| \leq \|f\|_\infty \int_0^A \Big| \theta_N(t) - \frac{1}{\alpha(2-\alpha)}\Big| dt.$$
For all $t\in[0,A]$, the previous integral comparison and the approximation of $x_{N,t}$ yield
\begin{align*}
	& \Big| \theta_N(t) - \frac{1}{\alpha(2-\alpha)}\Big| = \frac{1}{\alpha} \Big| \frac{1}{N^{2-\alpha}}\sum_{m=1}^{\floor*{x_{N,t}}}m^{1-\alpha} - \frac{1}{2-\alpha}\Big| \\
	\leq & \frac{1}{\alpha(2-\alpha)} \max\Big\{\Big| \frac{\floor*{x_{N,t}}^{2-\alpha}}{N^{2-\alpha}}+ (2-\alpha) \frac{\floor*{x_{N,t}}^{1-\alpha}}{N^{2-\alpha}}-1 \Big|, \, 1-\frac{\floor*{x_{N,t}}^{2-\alpha}}{N^{2-\alpha}} \Big\} \\
	\leq & \frac{1}{\alpha}\Big( \frac{1}{N} + \frac{t}{\alpha N^\alpha\phi(N)}\Big).
\end{align*}
We consequently get the estimate, for all $N\in\NN$,
\begin{equation}\label{eq:error_nu_N_poisson}
	| \nu_N^+(f) - \rho_\alpha \Leb_{\RR_+}(f) | \leq \frac{\|f\|_\infty A}{\alpha}\Big( \frac{A}{2\alpha N^\alpha\phi(N)}+\frac{1}{N} \Big).
\end{equation}
Summing the error terms from Lemma \ref{lem:linear_approx}, Equations \eqref{eq:error_mu_N_nu_N} and \eqref{eq:error_nu_N_poisson}, we finally get the effective convergence in the second case of Theorem \ref{th:effective_cv_correlations}: for all $N$ large enough so that $\phi(N) > \frac{A}{2^\alpha-1}$, there exists some real number $c_\alpha>0$, depending only on $\alpha$, such that
$$| \R_N^{\alpha,+}(f) - \rho_\alpha \Leb_{\RR_+}(f)| \leq \frac{c_\alpha}{2} (\|f\|_\infty + \|f'\|_\infty)A^3 \Big( \frac{\phi(N)}{N^{1-\alpha}} +  \frac{1}{N^\alpha \phi(N)} + \frac{1}{N} \Big).$$
(We used $\frac{c_\alpha}{2}$ in order to stick to the notations from Lemma \ref{lem:correl_symmetry}). An explicit example of such a constant is given by $c_\alpha = 2 \max\big\{c_\alpha', \frac{1+\alpha}{2\alpha},\frac{1}{2\alpha^2}\big\}$, where $c_\alpha'$ is defined in the proof of Lemma \ref{lem:linear_approx} as $c_\alpha'=\frac{2^{1+3(1-\alpha)}+2^{\frac{1}{\alpha}+4-3\alpha}(1-\alpha)}{3\alpha^2}$. Using Lemma \ref{lem:correl_symmetry}, the same $c_\alpha$ is an example of a constant for Theorem \ref{th:effective_cv_correlations}.
\medskip
	
\subsection[regime lambda]{Regime $\frac{\phi(N)}{N^{1-\alpha}} \underset{N\to\infty}{\longrightarrow}\lambda \in \RR_+^*$}

Let us first assume that $f \in C_c^1(\RR_+^*)$ (instead of $C_c^1(\RR)$) and choose $0<\e<1$ such that $\supp{f} \subset [\e,A]$. This lower bound on the support of $f$ will not be an obstacle, as the limiting measure will display some level repulsion property. We discuss how to pass to general test functions in $C_c^1(\RR)$ at the end of the proof.
	
For this third and final case, the previous estimate \eqref{eq:error_mu_N_nu_N} is not enough: it gives an error term that does not vanish as $N\to\infty$. This gives us a hint that the limit measure will be exotic in comparison to the ones from the two previous regimes. Let us temporarily use the explicit notation $\rho_\alpha=\rho_{\alpha,\lambda}$ in order to emphasize the dependence of the function $\rho_\alpha$ on $\lambda$. We first notice that, since the real function $x \mapsto \lambda x$ is (continuous and) proper, and thanks to the formula
$$\R_N^{\alpha} = \lambda (x \mapsto \lambda x)_* \Big( \frac{1}{\lambda \psi(N)} \sum_{1 \leq n,m \leq N} \Delta_{\frac{\phi(N)}{\lambda} (n^\alpha - m^\alpha)} \Big)$$
and the equality $\lambda (x\mapsto \lambda x)_* (\rho_{\alpha,1} \Leb_{\RR_+}) = \rho_{\alpha,1}\big(\frac{\cdot}{\lambda}\big) \Leb_{\RR_+} = \rho_{\alpha,\lambda}\Leb_{\RR_+}$, it is sufficient to prove Theorem \ref{th:cv_correlations} in the special case $\lambda=1$. For Theorem \ref{th:effective_cv_correlations}, we will discuss how the error term depends on $\lambda$ at the end of the proof. Henceforth, we assume that $\lambda=1$. As in the study of the regime $\frac{\phi(N)}{N^{1-\alpha}}\underset{N\to\infty}{\longrightarrow}0$, we use Lemmas \ref{lem:correl_symmetry} and \ref{lem:linear_approx} and study the behaviour of $(\mu_N^+)_{N\in\NN}$. Fix $N \geq 2$. Recall that $\psi(N)=\frac{N^{2-\alpha}}{\phi(N)} \underset{N\to\infty}{\sim}N$ in this case. Set $h_\alpha: x \mapsto x^{-(1-\alpha)}$ which is a diffeomorphism on $\RR_+^*$ with inverse $x\mapsto x^{-\frac{1}{1-\alpha}}$. We can then define the positive measure $\widetilde{\mu}_N^+$ on $\RR_+^*$ by the formula $\widetilde{\mu}_N^+=(h_\alpha^{-1})_*\mu_N^+$, that is
{\large\begin{equation}\label{eq:def_mu_tilde}
	\widetilde{\mu}_N^+ = \frac{1}{\psi(N)}\sum_{m=1}^{N-1}\sum_{p=1}^{N-m} \Delta_{m\frac{1}{(\alpha p \phi(N))^{\frac{1}{1-\alpha}}}} = \frac{1}{\psi(N)} \sum_{p=1}^{N-1} \sum_{m=1}^{N-p} \Delta_{m \frac{1}{(\alpha p \phi(N))^{\frac{1}{1-\alpha}}}}.
\end{equation}}
We thus see $\widetilde{\mu}_N^+$ as a weighted sum of Riemann sums with step denoted by
$$\delta_{N,p}=\frac{1}{(\alpha p \phi(N))^{\frac{1}{1-\alpha}}}.$$
We will compare it to the positive measure $\widetilde{\nu}_N^+$ defined by the equality
$$ \widetilde{\nu}_N^+ = \frac{1}{\psi(N)} \sum_{p=1}^{N-1} \frac{1}{\delta_{N,p}} \11_{[0,(N-p)\delta_{N,p}]} \Leb_{\RR_+}.$$
For that purpose, we will use Lemma \ref{lem:riemann_approx}, and thus need to understand thoroughly the quantity $\min\{A,(N-p)\delta_{N,p}\}$. We have the equivalence, for all $p\in\intbrackets*{1,N-1}$,
$$(N-p)\delta_{N,p} \geq A \iff g_N(p) \leq 0 \mbox{ where now } g_N : x \mapsto x - \frac{1}{\alpha A^{1-\alpha}} \frac{(N-x)^{1-\alpha}}{\phi(N)}.$$
A straightforward analysis of the function $g_N$ shows that it has a unique zero $x_N$ in $]0,N[$. Then we have the convergence $x_N \underset{N\to\infty}{\longrightarrow}\ell = \frac{1}{\alpha A^{1-\alpha}}$. We immediately get the following bound for the speed of convergence:
\begin{equation}\label{eq:estim_xN}
	\frac{x_N}{\ell} = \frac{(N-x_N)^{1-\alpha}}{\phi(N)} \leq \frac{N^{1-\alpha}}{\phi(N)}.
\end{equation}
Suppose first that $\alpha\neq \frac{1}{2}$ (i.e.~$\frac{1}{1-\alpha} \neq 2$). Because of the initial change of variable $h_\alpha$, we have to be cautious: when summing to get the total error term, we will apply Lemma \ref{lem:riemann_approx} to the function $\widetilde{f}=f \circ h_\alpha$. The inclusion $\supp{f} \subset [\e,A]$ yields $\supp{\widetilde{f}} \subset \big[ \frac{1}{A^{1-\alpha}}, \frac{1}{\e^{1-\alpha}} \big]$. Set $\widetilde{A}=\frac{1}{\e^{1-\alpha}}$. Applying Lemma \ref{lem:riemann_approx} to $\widetilde{f}$ with $\delta = \delta_{N,p}$, $M=N-p$ and $B=\widetilde{A}$, and using an integral comparison coming from the fact that the function $x \mapsto (N-x) x^{-\frac{1}{1-\alpha}}$ is nonincreasing on  $\RR_+^*$ (while being cautious of the case $\floor*{x_N}=0$ for the integral to be definite), we have
\begin{align*}
	& |\widetilde{\mu}_N^+(\widetilde{f}) - \widetilde{\nu}_N^+(\widetilde{f})|\\
	\leq & \frac{1}{\psi(N)} \sum_{p=1}^{N-1} \frac{1}{\delta_{N,p}} \Big| \int_0^{(N-p)\delta_{N,p}}\widetilde{f}(t) \, dt - \delta_{N,p} \sum_{m=1}^{N-p} \widetilde{f}(m\delta_{N,p}) \Big| \\
	\leq & \frac{\|\widetilde{f}'\|_\infty}{2\psi(N)} \sum_{p=1}^{N-1} \min\{\widetilde{A},(N-p)\delta_{N,p}\}\\
	\leq & \frac{\|f'\|_\infty}{2\psi(N)} \Big( \sum_{p=1}^{\floor*{x_N}} \widetilde{A} + \sum_{p=\floor*{x_N}+1}^{N-1}(N-p)\delta_{N,p} \Big) \\
	= & \frac{\widetilde{A}\|\widetilde{f}'\|_\infty}{2\psi(N)}\floor*{x_N} + \frac{\|\widetilde{f}'\|_\infty}{2\psi(N)} \frac{1}{(\alpha\phi(N))^{\frac{1}{1-\alpha}}} \sum_{p=\floor*{x_N}+1}^{N-1} \frac{N-p}{p^{\frac{1}{1-\alpha}}} \\
	\leq & \frac{\widetilde{A}\|\widetilde{f}'\|_\infty\ell}{2\psi(N)}\frac{N^{1-\alpha}}{\phi(N)} + \frac{\|\widetilde{f}'\|_\infty}{2\psi(N)} \frac{1}{(\alpha\phi(N))^{\frac{1}{1-\alpha}}} \Big( \frac{N}{(\floor*{x_N}+1)^{\frac{1}{1-\alpha}}} +  \int_{\floor*{x_N}+1}^N \frac{N-x}{x^{\frac{1}{1-\alpha}}} \, dx\Big)\\
	= & \frac{\widetilde{A}\|\widetilde{f}'\|_\infty\ell}{2N} +\frac{\|\widetilde{f}'\|_\infty N}{2\psi(N) (\alpha \phi(N) (\floor*{x_N}+1))^{\frac{1}{1-\alpha}}} \\
	& \quad + \frac{\|\widetilde{f}'\|_\infty}{2\psi(N)(\alpha\phi(N))^{\frac{1}{1-\alpha}}} \left[ \frac{N}{1- \frac{1}{1-\alpha}}x^{1-\frac{1}{1-\alpha}} - \frac{1}{2-\frac{1}{1-\alpha}}x^{2-\frac{1}{1-\alpha}} \right]_{x=\floor*{x_N}+1}^{x=N}
\end{align*}
where we used the formula $\psi(N)=\frac{N^{2-\alpha}}{\phi(N)}$. The expression between brackets is equal to
$$
\frac{(1-\alpha)^2}{\alpha(2\alpha-1)} N^{2 - \frac{1}{1-\alpha}} + \frac{1-\alpha}{\alpha}N(\floor*{x_N}+1)^{\frac{-\alpha}{1-\alpha}} + \frac{1-\alpha}{1-2\alpha}(\floor*{x_N}+1)^{\frac{1-2 \alpha}{1-\alpha}}.
$$
The function $x \mapsto x^{\frac{-\alpha}{1-\alpha}}$ is nonincreasing on $\RR_+^*$, and $x \mapsto x^{\frac{1-2\alpha}{1-\alpha}}$ is monotone (of monotony given by the sign of $\frac{1}{2}-\alpha$). By the inequality \eqref{eq:estim_xN}, for all $N$ large enough (depending only on $\alpha$, $A$, $\phi$), we have $ \ell \leq \floor*{x_N}+1 \leq \ell +2$, providing the estimates $(\floor*{x_N}+1)^{\frac{-\alpha}{1-\alpha}} \leq \ell^{\frac{-\alpha}{1-\alpha}}$ and $ (\floor*{x_N}+1)^{\frac{1-2\alpha}{1-\alpha}} \leq C(\ell,\alpha) = \ell^\frac{1-2\alpha}{1-\alpha} \mbox{ or } (\ell + 2)^{\frac{1-2\alpha}{1-\alpha}}$ (depending of the sign of $\frac{1}{2}-\alpha$).
Summing those error terms, recalling that $\phi(N) \sim N^{1-\alpha}$, hence $\psi(N) \sim N$, and using the inequalities $\frac{1}{A^{1-\alpha}} \leq \ell \leq \frac{1}{\alpha}$, we get the following bound
\begin{align*}
	|\widetilde{\mu}_N^+(\widetilde{f}) - \widetilde{\nu}_N^+(\widetilde{f})| & = \bigO_\alpha \Big( \|\widetilde{f}'\|_\infty \Big( \frac{\widetilde{A} \ell}{N} + \frac{1}{N \ell^{\frac{1}{1-\alpha}}} + \frac{1}{N^{\frac{1}{1-\alpha}}} + \frac{\ell^{\frac{-\alpha}{1-\alpha}}}{N} + \frac{C(\ell,\alpha)}{N^2} \Big) \Big) \\
	& = \bigO_\alpha \Big(\frac{\|\widetilde{f}'\|_\infty (\widetilde{A}\ell + \ell^{-\frac{1}{1-\alpha}})}{N}\Big) = \bigO_\alpha \Big(\frac{\|\widetilde{f}'\|_\infty (\widetilde{A}+A)}{N}\Big).
\end{align*}
For the case $\alpha=\frac{1}{2}$, the integration of $x\mapsto\frac{N-x}{x^{\frac{1}{1-\alpha}}}$ gives an extra error term of order $\log(N)$ coming with a factor $\frac{1}{N\phi(N)^{\frac{1}{1-\alpha}}} \sim \frac{1}{N^2}$, which keeps the result valid since $\frac{\log(N)}{N^2} \leq \frac{1}{N}$. Recalling the definitions $\widetilde{A}=\frac{1}{\e^{1-\alpha}}$ and $\widetilde{f}=f\circ h_\alpha$, we finally have
\begin{align*}
	|\widetilde{\mu}_N^+(f\circ h_\alpha) - \widetilde{\nu}_N^+(f\circ h_\alpha)| & = \bigO_\alpha \left( \frac{\| h_\alpha' \, f' \circ h_\alpha \|_\infty\big( \frac{1}{\e^{1-\alpha}}+A \big)}{N} \right) \\
	& = \bigO_\alpha \left( \frac{\| h_\alpha' \, f' \circ h_\alpha \|_\infty A}{N\e^{1-\alpha}} \right).\numberthis\label{eq:error_tilde_mu+_nu+}
\end{align*}
Our goal is now to find the limit of $(\widetilde{\nu}_N^+)_{N\in\NN}$ and to inverse the change of variable in order to get back to $(\R_N^{\alpha,+})_{N\in\NN}$. Let $t \in \RR_+^*$. The Radon-Nikodym derivative $\widetilde{\theta}_N$ of $\widetilde{\nu}_N^+$ (with respect to the Lebesgue measure) is given by
$$\widetilde{\theta}_N(t)=\frac{1}{\psi(N)} \sum_{p=1}^{N-1} \frac{1}{\delta_{N,p}} \11_{[0,(N-p)\delta_{N,p}]}(t).$$
Let us rewind using the change of variable $h_\alpha : x \mapsto x^{-(1-\alpha)}$. Set $\nu_N^+=(h_\alpha)_* \widetilde{\nu}_N^+$. Its Radon-Nikodym derivative $\theta_N$ verifies
$$\theta_N(t)=|(h_\alpha^{-1})'(t)| \; \widetilde{\theta}_N \circ h_\alpha^{-1}(t) = \frac{t^{-\frac{2-\alpha}{1-\alpha}}}{(1-\alpha) \psi(N)} \sum_{p=1}^{N-1} \frac{1}{\delta_{N,p}} \11_{[0,(N-p)\delta_{N,p}]}\big(t^{-\frac{1}{1-\alpha}}\big).$$
We have the equivalence, for all $p\in\intbrackets*{1,N-1}$,
$$(N-p)\delta_{N,p} \geq t^{-\frac{1}{1-\alpha}} \iff g_{N,t}(p) \leq 0 \mbox{ where now } g_{N,t} : x \mapsto x - \frac{t}{\alpha}\frac{(N-x)^{1-\alpha}}{\phi(N)}.$$
Once again, a direct study of these nondecreasing functions gives us the existence of a unique zero $x_{N,t}$ of $g_{N,t}$ in $]0,N[$. It verifies $x_{N,t} \underset{N\to\infty}{\longrightarrow}\frac{t}{\alpha}$ and its definition grants us the following estimation for its speed of convergence, valid for all $N \geq 2$,
\begin{equation}\label{eq:estim_x_N,t}
	\frac{t}{\alpha} \frac{N^{1-\alpha}}{\phi(N)} \Big(1-\frac{t}{\alpha N^\alpha \phi(N)} \Big)^{1-\alpha} \leq x_{N,t} \leq \frac{t}{\alpha} \frac{N^{1-\alpha}}{\phi(N)}.	
\end{equation}
This estimate is useful as it implies some uniform convergence, namely that for all compact subset $K$ in $\RR_+$, we have
\begin{equation}\label{eq:unif_conv_x_N,t}
	\sup_{t \in K} \Big|x_{N,t}-\frac{t}{\alpha}\Big| \underset{N\to\infty}{\longrightarrow}0.
\end{equation}
Using first the nonuniform version of this, we have the following pointwise convergence, for $t\in\RR_+-\alpha\NN$,
$$\theta_N(t) = \frac{\alpha^{\frac{1}{1-\alpha}}}{1-\alpha} \frac{\phi(N)^{\frac{1}{1-\alpha}}}{\psi(N)} t^{-\frac{2-\alpha}{1-\alpha}} \sum_{p=1}^{\floor*{x_{N,t}}} p^{\frac{1}{1-\alpha}} \underset{N\to\infty}{\longrightarrow} \theta_\infty(t) = \frac{\alpha^{\frac{1}{1-\alpha}}}{1-\alpha} t^{-\frac{2-\alpha}{1-\alpha}} \sum_{p=1}^{\floor*{\frac{t}{\alpha}}} p^{\frac{1}{1-\alpha}}.$$
Let $\theta_\infty : \RR_+ \to \RR_+$ denote the limit measurable function on $\RR_+$ in this (almost everywhere) convergence, which is the restriction to $\RR_+$ of the function $\rho_\alpha$ in Theorem \ref{th:cv_correlations} (for $\lambda=1$). In order to get an effective vague convergence, we first observe the inequality
\begin{align*}
	|\nu_N^+(f) - \theta_\infty \Leb_{\RR_+}(f)| & \leq \int_0^A |f(t)(\theta_N(t)-\theta_\infty(t))|\, dt \\
	& \leq \|f\|_\infty \frac{\alpha^{\frac{1}{1-\alpha}}}{1-\alpha} \frac{\phi(N)^{\frac{1}{1-\alpha}}}{\psi(N)} \int_0^A t^{-\frac{2-\alpha}{1-\alpha}} \; \Big|\sum_{p=1}^{\floor*{x_{N,t}}} p^{\frac{1}{1-\alpha}}-\sum_{p=1}^{\floor*{\frac{t}{\alpha}}} p^{\frac{1}{1-\alpha}}\Big| \, dt \\
	& \quad + \|f\|_\infty \Big| \frac{\phi(N)^{\frac{1}{1-\alpha}}}{\psi(N)} - 1 \Big| \int_0^A \theta_\infty(t) dt. \numberthis\label{eq:nu_theta_infty_exotic}
\end{align*}
For all $k\in\NN$, the function $\theta_\infty$ in bounded on the interval interval $[k\alpha,(k+1)\alpha[$. Since, by comparing to an integral, we have the convergence $\theta_\infty \underset{+\infty}{\longrightarrow} \frac{1}{\alpha(2-\alpha)}$, this proves that $\theta_\infty$ is bounded on $\RR_+$. As $\theta_\infty$ is defined using only the parameter $\alpha$, we have
\begin{align*}
	\|f\|_\infty \Big| \frac{\phi(N)^{\frac{1}{1-\alpha}}}{\psi(N)} - 1 \Big| \int_0^A \theta_\infty(t) dt & = \bigO_\alpha \Big( \|f\|_\infty A \Big| \frac{\phi(N)^{\frac{1}{1-\alpha}}}{\psi(N)}-1\Big| \Big)\\
	& = \bigO_\alpha \Big( \|f\|_\infty A \Big| \Big( \frac{\phi(N)}{N^{1-\alpha}} \Big)^{\frac{2-\alpha}{1-\alpha}} -1 \Big| \Big). \numberthis\label{eq:remainder_theta_infty_exotic}
\end{align*}
As the lower integer part function $\floor*{\cdot}$ is continuous on $\RR-\ZZ$, we know that, for all $t \in \RR_+-\alpha\NN$ and $N$ large enough depending on $t$ (and $\alpha$), we have the equality $\floor*{x_{N,t}}=\floor*{\frac{t}{\alpha}}$, meaning that $\theta_N(t)= \frac{\phi(N)^{\frac{1}{1-\alpha}}}{\psi(N)} \theta_\infty(t)$. We set
$$
\hat{\theta}_N = \frac{\psi(N)}{\phi(N)^{\frac{1}{1-\alpha}}} \theta_N.
$$
Thus the almost everywhere convergence of $(\hat{\theta}_N)_{N\in\NN}$ is stationary. However, it is not necessarily uniform as it can be much slower for $t$ close to $\alpha\NN$. Define two functions $\delta_-:t\mapsto t-\alpha \floor*{\frac{t}{\alpha}}$ and $\delta_+:t\mapsto \alpha\ceil*{\frac{t}{\alpha}}-t$, where $\ceil*{\cdot}$ denotes the upper integer part function. We use the speed of convergence of the sequences $(x_{N,t})_{N\in\NN}$ described in the inequalities \eqref{eq:estim_x_N,t} and we get, for all $t\in(\RR_+-\alpha\NN)\cap[0,A]$,

\begin{align*}
	x_{N,t} <\Big\lfloor \frac{t}{\alpha} \Big\rfloor+1 & \impliedby \frac{t}{\alpha} \frac{N^{1-\alpha}}{\phi(N)} < \Big\lfloor \frac{t}{\alpha} \Big\rfloor +1 \iff \frac{N^{1-\alpha}}{\phi(N)} < \frac{\floor*{\frac{t}{\alpha}}+1}{\frac{t}{\alpha}} = 1+\frac{\delta_+(t)}{t}\\
	\mbox{and } x_{N,t} \geq \Big\lfloor \frac{t}{\alpha} \Big\rfloor & \impliedby \frac{t}{\alpha} \frac{N^{1-\alpha}}{\phi(N)} \Big(1-\frac{t}{\alpha N^\alpha \phi(N)} \Big)^{1-\alpha} \geq \Big\lfloor \frac{t}{\alpha} \Big\rfloor \\
	\mbox{thus } x_{N,t} \geq \Big\lfloor \frac{t}{\alpha} \Big\rfloor & \impliedby \frac{N^{1-\alpha}}{\phi(N)} \Big(1-\frac{A}{\alpha N^\alpha \phi(N)} \Big)^{1-\alpha} \geq \frac{\floor*{\frac{t}{\alpha}}}{\frac{t}{\alpha}}=1-\frac{\delta_-(t)}{t}.
\end{align*}
Now let us study, for $N$ large enough independently on $t$, the proportion of $t\in[0,A]$ verifying both of these inequalities on $\delta_\pm(t)$. Let us define
$$X_N=\Big\{ t \in \RR_+ \, : \, \delta_+(t) \leq t\Big( \frac{N^{1-\alpha}}{\phi(N)}-1\Big) \mbox{ or } \delta_-(t) < t \Big( 1 - \frac{N^{1-\alpha}}{\phi(N)}\Big(1-\frac{A}{\alpha N^\alpha \phi(N)}\Big)^{1-\alpha} \Big)\Big\},$$
that is, the subset of $t$'s failing to verify at least one the two previous inequalities which were allowing to have $\hat{\theta}_N(t)= \theta_\infty(t)$. By definition of $\delta_-$ and $\delta_+$, the set $X_N$ is included in a union of intervals $I_k$ around each $k\alpha$, for $k\in\intbrackets*{0,\floor*{\frac{A}{\alpha}}}$, whose length is at most
$$ k\alpha \Big( \frac{N^{1-\alpha}}{\phi(N)}-1\Big) + (k+1)\alpha \Big( 1 - \frac{N^{1-\alpha}}{\phi(N)}\Big(1-\frac{A}{\alpha N^\alpha \phi(N)}\Big)^{1-\alpha} \Big) = \bigO_\alpha \Big( A \Big|\frac{N^{1-\alpha}}{\phi(N)}-1\Big| \Big).$$
As we will sum these lengths, it is important to notice that the right-hand side of the previous equality does not depend on $k$: there exists $c_\alpha''>0$ depending only on $\alpha$ such that, for $N$ large enough, depending on $\alpha$, $A$ and $\phi$, for all $k\in \intbrackets*{0,\floor*{\frac{A}{\alpha}}}$, we have the inequality
$$
\Leb_\RR(I_k) \leq c_\alpha'' A \Big( \frac{N^{1-\alpha}}{\phi(N)}-1 \Big).
$$
In order to also get some upper bound on $|\hat{\theta}_N - \theta_\infty|$ on $X_N$, we use the uniform convergence property \eqref{eq:unif_conv_x_N,t}: we know that there exists $N_0 \in \NN$ (depending on $\alpha$ and $A$) such that for all $N\geq N_0$ and for all $t \in [0,A]$, we have the inequality $\big|\floor*{x_{N,t}} - \floor*{\frac{t}{\alpha}}\big| \leq 1$. We also notice that $\hat{\theta}_N = \theta_\infty$ near $0$. More precisely, there exists some integer $N_1 \geq N_0$ such that, for all $N \geq N_1$, we have the inequality $\frac{N^{1-\alpha}}{\phi(N)} < 2$, hence $x_{N,t} < 2 \frac{t}{\alpha}$ thanks to the right-hand side in the inequalities \eqref{eq:estim_x_N,t}. For such integers $N$, we have the equality $\hat{\theta}_N = \theta_\infty(=0)$ on $[0,\frac{\alpha}{2}]$. This equality can be understood as the level repulsion phenomenon for $\theta_N$. We can now bound from above the integral of $| \hat{\theta}_N-\theta_\infty|$. Indeed, for all $N \geq N_1$, we have

\begin{align*}
	& \int_0^A t^{-\frac{2-\alpha}{1-\alpha}} \Big|\sum_{p=1}^{\floor*{x_{N,t}}} p^{\frac{1}{1-\alpha}}-\sum_{p=1}^{\floor*{\frac{t}{\alpha}}} p^{\frac{1}{1-\alpha}}\Big| \, dt = \int_{[\frac{\alpha}{2},A] \cap X_N} t^{-\frac{2-\alpha}{1-\alpha}} \Big|\sum_{p=1}^{\floor*{x_{N,t}}} p^{\frac{1}{1-\alpha}}-\sum_{p=1}^{\floor*{\frac{t}{\alpha}}} p^{\frac{1}{1-\alpha}}\Big| \, dt \\
	\leq & \int_{[\frac{\alpha}{2},A] \cap X_N} t^{-\frac{2-\alpha}{1-\alpha}} \Big( x_{N,t}^{\frac{1}{1-\alpha}} + \Big( \frac{t}{\alpha} \Big)^{\frac{1}{1-\alpha}} \Big) \, dt \\
	\leq & \sum_{k=1}^{\floor*{\frac{A}{\alpha}}} \int_{[\frac{\alpha}{2},A]\cap I_k} t^{-\frac{2-\alpha}{1-\alpha}} \Big( \Big( 2\frac{t}{\alpha} \Big)^{\frac{1}{1-\alpha}} + \Big( \frac{t}{\alpha} \Big)^{\frac{1}{1-\alpha}} \Big) \, dt
	\leq \frac{2^{\frac{1}{1-\alpha}}+1}{\alpha^{\frac{1}{1-\alpha}}} \sum_{k=1}^{\floor*{\frac{A}{\alpha}}} \int_{[\frac{\alpha}{2},A]\cap I_k} \frac{dt}{t} \\
	\leq & \frac{1}{\alpha^{\frac{1}{1-\alpha}}} \sum_{k=1}^{\floor*{\frac{A}{\alpha}}} \frac{4.2^{\frac{1}{1-\alpha}}}{\alpha} c_\alpha'' A \Big|\frac{N^{1-\alpha}}{\phi(N)}-1\Big| = \bigO_\alpha \left(A^2 \Big|\frac{N^{1-\alpha}}{\phi(N)}-1\Big| \right).
\end{align*}
By Equations \eqref{eq:nu_theta_infty_exotic} and \eqref{eq:remainder_theta_infty_exotic}, and since $A \geq 1$, this gives the final error term for the vague convergence of $(\nu_N^+)_{N\in\NN}$:
\begin{align*}
	|\nu_N^+(f) - \theta_\infty \Leb_{\RR_+}(f)|
	& = \bigO_\alpha \left(\|f\|_\infty A^2 \Big|\frac{N^{1-\alpha}}{\phi(N)}-1\Big| + A \Big| \Big( \frac{\phi(N)}{N^{1-\alpha}} \Big)^{\frac{2-\alpha}{1-\alpha}} -1 \Big| \right) \\
	& = \bigO_\alpha \left( \|f\|_\infty A^2 \Big| \Big( \frac{\phi(N)}{N^{1-\alpha}} \Big)^{\frac{2-\alpha}{1-\alpha}} -1 \Big| \right). \numberthis\label{eq:error_nu_+_theta}
\end{align*}
Recalling the definition of $h_\alpha : x \mapsto \frac{1}{x^{1-\alpha}}$, hence $|h_\alpha'| : x \mapsto (1-\alpha) \frac{1}{x^{2-\alpha}}$, and summing the error terms from Lemma \ref{lem:linear_approx} and Equations \eqref{eq:error_tilde_mu+_nu+}, \eqref{eq:error_nu_+_theta}, we finally obtain
	\begin{align*}
	& |\R_N^{\alpha,+}(f)-\theta_\infty \Leb_{\RR_+}(f)| \\
	\leq & |\R_N^{\alpha,+}(f) - \mu_N^+(f)| + |\mu_N^+(f) - \nu_N^+(f)| + |\nu_N^+(f) - \theta_\infty \Leb_{\RR_+}(f)| \\
	= & |\R_N^{\alpha,+}(f) - \mu_N^+(f)| + |(h_\alpha)_*\widetilde{\mu}_N^+(f) - (h_\alpha)_*\widetilde{\nu}_N^+(f)| + |\nu_N^+(f) - \theta_\infty \Leb_{\RR_+}(f)| \\
	= & \bigO_\alpha \left( \frac{A^3 \|f'\|_\infty}{N^\alpha \phi(N)}\right) + |\widetilde{\mu}_N^+(f \circ h_\alpha) - \widetilde{\nu}_N^+(f \circ h_\alpha)| + \bigO_\alpha \left( \|f\|_\infty A^2\Big| \Big( \frac{\phi(N)}{N^{1-\alpha}} \Big)^{\frac{2-\alpha}{1-\alpha}} -1 \Big| \right) \\
	= & \bigO_\alpha \left( \frac{A^3 \|f'\|_\infty}{N^\alpha \phi(N)} + \frac{\| h_\alpha' f' \circ h_\alpha \|_\infty A}{N\e^{1-\alpha}} + A^2 \|f\|_\infty \Big| \Big( \frac{\phi(N)}{N^{1-\alpha}} \Big)^{\frac{2-\alpha}{1-\alpha}} -1 \Big| \right) \\
	= & \bigO_\alpha \left( \frac{A^3 \|f'\|_\infty}{N} + \frac{\big|h_\alpha'\big(A^{-\frac{1}{1-\alpha}}\big)\big| \|f'\|_\infty A}{N\e^{1-\alpha}} + A^2 \|f\|_\infty \Big| \Big( \frac{\phi(N)}{N^{1-\alpha}} \Big)^{\frac{2-\alpha}{1-\alpha}} -1 \Big| \right) \\
	= & \bigO_\alpha \left( \frac{A^3 \|f'\|_\infty}{N} + \frac{A^{\frac{3-2\alpha}{1-\alpha}}\|f'\|_\infty}{N \e^{1-\alpha}} + A^2 \|f\|_\infty \Big| \Big( \frac{\phi(N)}{N^{1-\alpha}} \Big)^{\frac{2-\alpha}{1-\alpha}} -1 \Big| \right).
\end{align*}

Now let us drop the assumption on the existence of some positive lower bound $\e$ for $\supp{f}$: let $f \in C_c^1(\RR)$ and choose $A>1$ such that $\supp{f} \subset [-A,A]$. We remark that both the positive measure $\theta_\infty \Leb_{\RR_+}$ and the measures $\R_N^{\alpha,+}$, for $N\in\NN-\{0\}$, display some level repulsion property. Indeed, the function $\theta_\infty$ vanishes on $[0,\alpha[$, and Equation \eqref{eq:empirical_level_repulsion} implies that, for all $N\in\NN$, we have $\supp{\R_N^{\alpha,+}} \subset \big[ \frac{\alpha \phi(N)}{2N^{1-\alpha}}, +\infty \big[$.  For all $N$ large enough so that $\frac{\phi(N)}{N^{1-\alpha}} > \frac{1}{2}$, we thus have both inclusions
$$\supp(\theta_\infty \Leb_{\RR_+}) \subset [\alpha, +\infty [ \mbox{ and } \supp{\R_N^{\alpha,+}} \subset \big[ \frac{\alpha}{4}, +\infty \big[.$$

Set $\e = \frac{\alpha}{8}$. By a standard smoothing process, we know there exists a function $g \in C_c^1(\RR)$ verifying
\begin{enumerate}
	\item the functional equality $g=f$ (and hence $g'=f'$) on the interval $[2\e,A]$,
	\item the inclusion $\supp{g} \subset [\e,A]$,
	\item the inequality $\|g\|_\infty \leq \|f\|_\infty$,
	\item and the inequality $\|g'\|_\infty \leq \|f'\|_\infty + 2\frac{\|f\|_\infty}{\e}$.
\end{enumerate}
For $N$ large enough, the interval $[2\e,A]$ contains the support of both measures $\theta_\infty \Leb_{\RR_+}$ and $\R_N^{\alpha,+}$. Thus, the approximation $g$ of $f$ grants us the asymptotic upper bound
{\small
\begin{align*}
	& |\R_N^{\alpha,+}(f) - \theta_\infty \Leb_{\RR_+}(f)| = |\R_N^{\alpha}(g) - \theta_\infty \Leb_{\RR_+}(g)| \\
	= & \bigO_\alpha \left( \frac{A^3 \|g'\|_\infty}{N} + \frac{A^{\frac{3-2\alpha}{1-\alpha}}\|g'\|_\infty}{N \e^{1-\alpha}} + A^2 \|g\|_\infty \Big| \Big( \frac{\phi(N)}{N^{1-\alpha}} \Big)^{\frac{2-\alpha}{1-\alpha}} -1 \Big|  \right)\\
	= & \bigO_\alpha \left( \frac{\big(A^3+A^{\frac{3-2\alpha}{1-\alpha}}) (\|f'\|_\infty+\|f\|_\infty)}{N} + A^2 \|f\|_\infty \Big| \Big( \frac{\phi(N)}{N^{1-\alpha}} \Big)^{\frac{2-\alpha}{1-\alpha}} -1 \Big| \right).
\end{align*}
}
Since $3 < \frac{3-2\alpha}{1-\alpha}$, this concludes the proof in the case $\lambda=1$. For the general case, we use the notation $f_\lambda : x \mapsto f(\lambda x)$. Using again the notation $\rho_{\alpha,\lambda}$ to underline the dependence of $\rho_\alpha$ on $\lambda$, we have
\begin{align*}
	& |\R_N^\alpha(f)-\rho_{\alpha,\lambda} \Leb_\RR(f)|\\
	= & \Big| \lambda (x \mapsto \lambda x)_* \Big( \frac{1}{\lambda \psi(N)} \sum_{1 \leq n \neq m \leq N} \Delta_{\frac{\phi(N)}{\lambda}(n^\alpha - m^\alpha)}\Big) (f) - \lambda (x \mapsto \lambda x)_* (\rho_{\alpha,1} \Leb_\RR)(f) \Big| \\
	= & \lambda \Big| \frac{1}{\lambda \psi(N)} \sum_{1 \leq n \neq m \leq N} f_\lambda \Big(\frac{\phi(N)}{\lambda} (n^\alpha - m^\alpha)\Big) - \rho_{\alpha,1} \Leb_\RR (f_\lambda) \Big| \\
	= & \lambda \bigO_\alpha \left( \frac{A^{\frac{3-2\alpha}{1-\alpha}} (\|f_\lambda'\|_\infty+\|f_\lambda\|_\infty)}{N} + A^2 \|f_\lambda \|_\infty \Big| \Big( \frac{\phi(N)}{\lambda N^{1-\alpha}} \Big)^{\frac{2-\alpha}{1-\alpha}} -1 \Big| \right) \\
	= & \bigO_\alpha \left( \frac{A^{\frac{3-2\alpha}{1-\alpha}} (\lambda^2 \|f'\|_\infty+\lambda \|f \|_\infty)}{N} + A^2 \lambda \|f \|_\infty \Big| \Big( \frac{\phi(N)}{\lambda N^{1-\alpha}} \Big)^{\frac{2-\alpha}{1-\alpha}} -1 \Big| \right).
\end{align*}
This proves Theorem \ref{th:effective_cv_correlations} under the third regime, i.e.~assuming $\lambda \in \RR_+^*$, and finally concludes the proof of Theorem \ref{th:effective_cv_correlations}.\qed

\section{Unfolding technique}\label{sec:unfolding}
In this section, we find again the pair correlation function $\rho_\alpha$ in the exotic case with the scaling factor $\phi: N \mapsto N^{1-\alpha}$ and the renormalization factor $\psi: N\mapsto N$, using some interpretation of the unfolding technique (see for instance \cite[§~2.1]{marklof2002quadformII}), though extending it to general scaling factors and obtaining the error terms might require even more work than in the previous sections. In order to use that technique, we replace the positive part of the empirical pair correlation measure
$$\R_N^{\alpha,+} = \frac{1}{N} \sum_{0<m<n\leq N} \Delta_{N^{1-\alpha}(n^\alpha-m^\alpha)}$$
by the following positive measure on $]0,1]^2 \times \RR_+^*$,
$$\widetilde{\R}_N^{\alpha,+} = \frac{1}{N} \sum_{0<m<n\leq N} \Delta_{(\frac{n}{N}, \frac{m}{N}, \, N^{1-\alpha}(n^\alpha-m^\alpha))}.$$
Finding some limit measure for the sequence $(\widetilde{\R}_N^{\alpha,+})_{N\in\NN^*}$ then pushing forward on the third component of $]0,1]^2 \times \RR_+^*$ will imply convergence of $(\R_N^{\alpha,+})_{N\in\NN^*}$ and allow us to compute the limit. We look for a change of variables $h_\alpha$ from $]0,1]^2 \times \RR_+^*$ to itself, of the form
$$ h_\alpha : (x,y,t) \mapsto (x,y, g_\alpha (x,y)t) $$
where $g_\alpha$ is a function not depending on $N$ and verifying, for all $N,m,n$, the formula $g_\alpha(\frac{n}{N}, \frac{m}{N}) N^{1-\alpha} (n^\alpha - m^\alpha) = n-m$. Such a function is given by
$$
g_\alpha : (x,y) \mapsto \left\{
\begin{array}{cl}
	\frac{x-y}{x^\alpha-y^\alpha} & \mbox{if } x \neq y,\vspace{2mm} \\
	\frac{x^{1-\alpha}}{\alpha} & \mbox{if } x=y.	
\end{array} 
\right.
$$
The definition of $g_\alpha$ on the case $x=y$ allows the function to be continuous on $]0,1]^2$. We denote by $\Leb_{\rm diag}$ the probability measure on (the diagonal of) the square $]0,1]^2$ defined by
$$ \forall f \in C^0(\,]0,1]^2), \, \Leb_{\rm diag}(f) = \int_0^1 f(t,t) \, dt. $$
Applying the change of variables $h_\alpha$, we have
\begin{align*}
(h_\alpha)_* \widetilde{\R}_N^{\alpha,+} = \frac{1}{N} \sum_{0<m<n\leq N} \Delta_{(\frac{n}{N}, \frac{m}{N}, \, n-m)} & = \frac{1}{N} \sum_{p=1}^{N-1} \sum_{m=1}^{N-p} \Delta_{(\frac{m}{N} + \frac{p}{N}, \frac{m}{N}, \, p)} \\ & \underset{N\to\infty}{\weakstar} \sum_{p=1}^{+\infty} \Leb_{\rm diag} \otimes \Delta_p
\end{align*}
(for this last convergence, we use the fact that $\frac{p}{N}$ is negligible with respect to $\frac{m}{N}$ since $p$ stays in a bounded interval when the measure is evaluated on a function in $C_c^0(\,]0,1]^2 \times \RR)$). The function $h_\alpha$ is an homeomorphism, in particular $h_\alpha^{-1}$ is continuous and proper, thus we have
the vague convergence
$$ \widetilde{\R}_N^{\alpha,+} \underset{N\to\infty}{\weakstar} \widetilde{\R}_\infty^{\alpha,+} = (h_\alpha^{-1})_* \big( \sum_{p=1}^{+\infty} \Leb_{\rm diag} \otimes \Delta_p \big) = \sum_{p=1}^{+\infty} \int_0^1 \Delta_{(t,t,\frac{p}{g_\alpha(t,t)})} \, dt.$$
We recall the formula $g_\alpha(t,t) = \frac{t^{1-\alpha}}{\alpha}$. Now, we project on the third component of this limit measure by taking a function of one variable $f\in C_c^0(\RR)$ and we compute a density, by using the change of variable $u=\frac{\alpha p}{t^{1-\alpha}}$, as follows
\begin{align*}
	\widetilde{\R}_\infty^{\alpha,+}(1\otimes f) & = \sum_{p=1}^{+\infty} \int_0^1 1 \times f\big(\frac{\alpha p}{t^{1-\alpha}}\big) \, dt = \sum_{p=1}^{+\infty} \int_{\alpha p}^{+\infty} f(u) \frac{u^{-\frac{2-\alpha}{1-\alpha}}}{1-\alpha}(\alpha p)^{\frac{1}{1-\alpha}} \, du \\
	& = \int_0^{+\infty}f(u) \frac{\alpha^{\frac{1}{1-\alpha}}}{1-\alpha} u^{-\frac{2-\alpha}{1-\alpha}} \sum_{p=1}^{\floor*{\frac{u}{\alpha}}} p^{\frac{1}{1-\alpha}} \, du = \int_0^{+\infty} f(u) \rho_\alpha(u) \, du
\end{align*}
where $\rho_\alpha$ is the pair correlation function defined in Theorem \ref{th:cv_correlations_simplified}. Let us denote by $\pi_3:[0,1]^2 \times \RR_+^* \to \RR_+^*$ the projection on the third coordinate. The above computation implies that $(\pi_3)_* \widetilde{\R}_\infty^{\alpha,+} = \rho_\alpha \Leb_{\RR_+}$. Since $\pi_3$ is a continuous and proper function, we finally deduce the convergence of $(\R_N^{\alpha,+})_{N\in\NN^*}$, namely
$$\R_N^{\alpha,+} = (\pi_3)_*\widetilde{\R}_N^{\alpha,+} \underset{N\to\infty}{\weakstar} (\pi_3)_* \widetilde{\R}_\infty^{\alpha,+} = \rho_\alpha \Leb_{\RR_+}.$$

\AtNextBibliography{\small}
\printbibliography[heading=bibintoc, title={References}]

\end{document}